\documentclass[12pt]{amsart}
\usepackage{amssymb,latexsym}
\usepackage{amsfonts}
\usepackage{amsmath}
\usepackage[colorlinks,linkcolor=blue,anchorcolor=blue,citecolor=blue]{hyperref}
\usepackage{algorithm}
\usepackage{enumerate}
\usepackage{algpseudocode}
\usepackage{verbatim}
\usepackage{graphicx}

\newcommand\F{{\mathbb F}}

\newcommand\ord{\mathrm{ord}}
\newcommand\lcm{\mathrm{lcm}}

\newcommand\Tr{\mathrm{Tr}}

\newtheorem{theorem}{Theorem}[section]
\newtheorem{lemma}[theorem]{Lemma}

\newtheorem{corollary}[theorem]{Corollary}

\theoremstyle{definition}

\newtheorem{remark}[theorem]{Remark}
\newtheorem{example}[theorem]{Example}

\numberwithin{equation}{section}

\begin{document}

\title[]{On the generalized Fibonacci sequence of polynomials over finite fields}

\author{Zekai Chen}
\address{School of Mathematical Sciences, South China Normal University, Guangzhou 510631, China}
\email{chenzk@m.scnu.edu.cn}

\author{Min Sha}
\address{School of Mathematical Sciences, South China Normal University, Guangzhou 510631, China}
\email{min.sha@m.scnu.edu.cn}

\author{Chen Wei}
\address{School of Mathematical Sciences, South China Normal University, Guangzhou 510631, China}
\email{20202232017@m.scnu.edu.cn}







\begin{abstract}
In this paper, as an analogue of the integer case, we study detailedly the period and the rank of the generalized Fibonacci sequence of polynomials over a finite field modulo an arbitrary polynomial. We establish some formulas to compute them, and we also obtain some properties about the quotient of the period and the rank. 
We find that the polynomial case is much more complicated than the integer case. 
\end{abstract}

\maketitle

\section{Introduction}

\subsection{Motivation}

The famous Fibonacci sequence $\{0, 1, 1, 2, 3, 5,\ldots\}$ has been studied for a very long time and from various aspects, 
and its study is still a very active area in number theory. 
One natural topic is to study the properties of the Fibonacci sequence under a modulus. 
It is well-known that the Fibonacci sequence is periodic modulo any integer $m$; see \cite{Hall, Rob, Vin, Wall} for more properties. 
In addition, mathematicians have also studied the periodic properties of the generalized Fibonacci sequences 
(that is, the second-order linear recurrence sequences with initial terms $0$ and $1$) under a modulus; 
see \cite{Car, Eng, Gup, Li, Lucas, Ren, Vel, Vel2, Ward} for more details. 
The two quantities studied most extensively about periodic properties are the period and the rank. 

As an analogue of the integer case, in this paper we study the periodic properties of 
 the generalized Fibonacci sequence of polynomials over finite fields. 
 
Note that the generalized Fibonacci sequences are second-order linear recurrence sequences. 
Another motivation for our work is the study of the periodicity of linear recurrence sequences over number fields; 
see Chapter 3 of \cite{EPSW} for a detailed exposition and references therein. 
One application of this study is constructing pseudorandom number generators used in cryptography. 
 
We also remark that  the generalized Fibonacci sequence of polynomials over the the rational numbers (whose definition has the same manner as the one below)
has been introduced recently and studied extensively (such as, combinatorics identities, divisibility properties, irreducibility properties, and so on); 
see \cite{FHM, FHR, FMM, FS} and the references therein.

\subsection{Our situation}

Let $\F_q$ be the finite field of $q$ elements, where $q = p^l$ for some prime $p$ and some integer $l > 0$. 
Let $\F_q[x]$ be the polynomial ring of one variable over $\F_q$. 
Throughout this paper, we fix two polynomials $a, b \in \F_q[x]$ with $b \ne 0$. 
The \textit{generalized Fibonacci sequence}, denoted by $F$, of polynomials over $\F_q$ and with respect to $a$ and $b$, is defined to be 
$$
F_0 = 0, F_1 = 1, F_n = aF_{n-1} + bF_{n-2}, n=2, 3, \ldots. 
$$
The famous sequence of \textit{Fibonacci polynomials} corresponds to the case when $a=x$ and $b=1$. 
We remark that in \cite{Ren} the integer analogue of this sequence is called the $(a,b)$-Fibonacci sequence. 

Let $M \in \F_q[x]$ be a non-constant polynomial which is relatively prime to $b$ (that is, $\gcd(b,M)=1$). 
Then, the sequence $F$ becomes a periodic sequence modulo $M$. 
In this paper, we study the period and the rank of the sequence $F$ modulo $M$. 

The \textit{rank of apparition}, or simply \textit{rank} of $F$ is the least positive integer $n$ such that $F_n \equiv 0 \pmod{M}$, 
and we denote the rank by $\alpha(M)$. 

We denote the \textit{period} of $F \pmod{M}$ by $\pi(M)$. 
That is, $\pi(M)$ is the least positive integer $n$ such that $F_n \equiv 0 \pmod{M}$ and $F_{n+1} \equiv 1 \pmod{M}$. 

We also define $\beta(M) = \pi(M)/\alpha(M)$. We shall see soon that $\beta(M)$ is always an integer.  
Then, $\beta(M)$ is exactly the number of zero terms in one (minimal) period of the sequence $F$ modulo  $M$. 

The aim of this paper is to study the computations and the arithmetic properties of the three functions: $\alpha(M), \pi(M), \beta(M)$. 
Some of the techniques are similar as the integer case (such as the matrix methods, which was first introduced by Robinson \cite{Rob}, see also \cite{Ren}), 
and some of them are quite different from the integer case (for instance, see the proof of Theorem~\ref{thm:ei'}). 
In a word, the polynomial case is much more complicated than the integer case (for instance, see Theorem~\ref{thm:ei'}). 

We collect some preliminary results in Section~\ref{sec:pre}. 
In Section~\ref{sec:M} we transfer the problem about a general polynomial $M$ to that about $P^e$, a power of an irreducible polynomial. 
We then study  $\alpha(P^e)$ and  $\pi(P^e)$ in Section~\ref{sec:Pe} (see Theorems~\ref{thm:alpha} and \ref{thm:pi}) and essentially reduce the problem to compute 
$\alpha(P)$ and  $\pi(P)$, which is considered in Section~\ref{sec:P}. 
Finally, in Section~\ref{sec:beta} we study the function $\beta$.

\section{Preliminaries}   \label{sec:pre}

\subsection{Matrix representation}
First we denote by $U$ the matrix 
$$
U=\begin{bmatrix}
    0&1\\
    b&a
\end{bmatrix}, 
$$ 
and let $I$ be the identity matrix. 
Similar as the integer case, it is easily confirmed by induction that for any integer $n \ge 1$, we have 
$$
U^n=\begin{bmatrix}
    bF_{n-1}&F_n\\
    bF_n&F_{n+1}
\end{bmatrix}.
$$

Note that if $F_n \equiv 0 \pmod{M}$, then $F_{n+1} = aF_n + bF_{n-1} \equiv bF_{n-1} \pmod{M}$. 
So,  by definition we directly have $U^{\pi(M)} \equiv I \pmod{M}$, and $U^{\alpha(M)} \equiv sI \pmod{M}$ for some polynomial $s \in \F_q[x]$ 
satisfying $\gcd(s,M)=1$. 
In other words, $U^{\alpha(M)}$ is a scalar invertible matrix modulo $M$. 

Indeed, for the above polynomial $s$,  if $\gcd(s,M) \ne 1$, then there is an irreducible polynomial $P \in \F_q[x]$ satisfying $P \mid \gcd(s,M)$ such that 
$U^{\alpha(M)} \equiv 0 \pmod{P}$, and this together with $\gcd(b,M)=1$ gives that $F_n \equiv 0 \pmod{P}$ for any integer $n \ge 0$, 
which contradicts with $F_1 = 1$. Hence, we must have $\gcd(s,M)=1$. 

Notice that the determinant $\det{U}=-b$. 
Thus, we obtain
$$
(-b)^{\pi(M)}=(\det U)^{\pi(M)}=\det U^{\pi(M)} \equiv 1\pmod{M},
$$ 
which gives
$$ 
\ord_M (-b) \mid \pi(M). 
$$ 
Here $\ord_M (-b)$ is the multiplicative order of $-b$ modulo $M$. 
This implies that if the characteristic $p$ is odd and $b=1$, then $\pi(M)$ is even. 

We claim that 
\begin{equation}  \label{eq:UnI}
U^{n} \equiv I \pmod{M} \iff \pi(M) \mid n.
\end{equation}
Indeed, if $\pi(M) \mid n$, then there exists a positive integer $k$ such that $n=k \pi (M) $, so $U^n=U^{k \pi(M)}=(U^{\pi(M)})^k \equiv I^k \equiv  I \pmod{M}$.
Conversely, if $U^n \equiv I \pmod{M}$, by contradiction we assume that $\pi(M) \nmid n$, which means that there exist two positive integers $d,r$ such that $n=\pi(M)d + r$ with $0 < r < \pi(M)$. 
Then $U^n=U^{\pi(M)d + r}=(U^{\pi(M)})^d U^r$. Due to $U^n \equiv U^{\pi(M)} \equiv I \pmod{M}$, we have $U^r \equiv I  \pmod{M}$.
Since $0 < r < \pi(M)$, this contradicts the choice of $\pi(M)$. 
So, we must have $\pi(M) \mid n$. 

We remark that \eqref{eq:UnI} is equivalent to that 
\begin{equation}  \label{eq:Fn01-M}
F_n \equiv 0 \text{ and } F_{n+1} \equiv 1 \pmod{M} \iff \pi(M) \mid n.
\end{equation}

Similarly, we can obtain 
\begin{equation}  \label{eq:Uns}
U^{n} \text{ is a scalar matrix modulo $M$}  \iff \alpha(M) \mid n, 
\end{equation}
which is equivalent to that 
\begin{equation}\label{eq:Fn0-M}
F_n \equiv 0  \pmod{M} \iff \alpha(M) \mid n.
\end{equation}

Suppose that $U^{\alpha(M)} \equiv sI \pmod{M}$ for some polynomial $s \in \F_q[x]$. 
Then, recalling $\beta(M)= \pi(M)/\alpha(M)$, 
we have 
$$
s^{\beta(M)}I=(sI)^{\beta(M)} \equiv U^{\alpha(M)\beta(M)} \equiv U^{\pi(M)} \equiv I \pmod{M},
$$ 
and so 
\begin{equation}  \label{eq:betas}
\beta(M)= \ord_M(s), \quad \textrm{equivalently,} \quad  \pi(M) = \alpha(M)\ord_M(s).
\end{equation}

\subsection{Arithmetic properties}

Here, we reproduce some arithmetic properties of the sequence $\{F_n\}$ following the integer case. 

Let $\overline{\F}_q$ be the algebraic closure of $\F_q$. 
Let $\F_q(x)$ be the quotient field of $\F_q[x]$, 
    and let $\overline{\F_q(x)}$ be the algebraic closure of $\F_q(x)$. 

As usual, the characteristic polynomial of the sequence $\{F_n\}$ is defined by 
$$
f(X)=X^2-aX-b,
$$
whose discriminant is $\Delta = a^2+4b$. 
Note that $f(X)$ is also the characteristic polynomial of the matrix $U$. 

Recall that $b \ne 0$, and $p$ is the characteristic of $\F_q$. 
We first establish a formula for $F_n$. 

\begin{lemma}  \label{lem:Fn}
Let $\lambda_1, \lambda_2$ be the two roots of $f(X)$ in $\overline{\F_q(x)}$. 
Then, for any integer $n \ge 0$, we have 
 \begin{equation*}
     F_n = 
            \begin{cases}
                 n\lambda_1^{n-1} & \text{if $\lambda_1 = \lambda_2$ $($that is $\Delta = 0)$}, \\
                \dfrac{\lambda_1^n-\lambda_2^n}{\lambda_1-\lambda_2} & \text{if $\lambda_1 \ne \lambda_2$ $($that is $\Delta \ne 0)$}. 
            \end{cases}
    \end{equation*}
\end{lemma}

\begin{proof}
We first note that both $\lambda_1$ and $\lambda_2$ are non-zero, because $b \ne 0$. 

Suppose $\lambda_1 \ne \lambda_2$ (that is $\Delta \ne 0$). 
    Then, $F_n=c_1 \lambda_1^n+c_2 \lambda_2 ^n$
    for some $c_1,c_2 \in \overline{\F_q(x)},n=0,1,2,\ldots$.
    Since $F_0=0,F_1=1$, 
    we get
    \begin{equation*}
        c_1+c_2=0, \qquad   c_1\lambda_1+c_2\lambda_2=1.
    \end{equation*}
This yields 
    \begin{equation*}
        c_1=\frac{1}{\lambda_1-\lambda_2}, \qquad  c_2=-c_1.
    \end{equation*}
    Thus, 
    \begin{equation*}
    F_n=\frac{\lambda_1^n-\lambda_2^n}{\lambda_1-\lambda_2},\quad n=0,1,2,\ldots.
    \end{equation*}

Now,  we suppose $\lambda_1 = \lambda_2$ (that is $\Delta = 0$). 
    Then, $F_n=(c_1+c_2 n) \lambda_1^n$ for some  $c_1,c_2 \in \overline{\F_q(x)},n=0,1,2,...$.
    Since $F_0=0,F_1=1$, 
    we obtain 
    \begin{equation*}
         c_1=0, \qquad c_2=\frac{1}{\lambda_1}.
    \end{equation*}
    Thus,
    \begin{equation*} 
        F_n=n\lambda_1^{n-1}.
    \end{equation*} 
\end{proof}

In Lemma~\ref{lem:Fn}, one needs to compute the two roots $\lambda_1, \lambda_2$ of the equation $X^2 - aX -b = 0$. 
When the characteristic $p$ is odd, the well-known quadratic formula gives the two roots: 
$$
\frac{a \pm \sqrt{a^2 + 4b}}{2}. 
$$
When $p=2$, the equation $X^2 - aX -b = 0$ becomes $X^2 + aX +b = 0$. If further $a=0$, then the two roots are the same and equal to $\sqrt{b}$. 
If $a \ne 0$, then let $c = b/a^2$ and let $R(c)$ be a root of the equation $X^2 + X + c = 0$, 
and then one can check directly that the two roots of $X^2 + aX +b = 0$ are:
$$
aR(c), \quad a(R(c) +1). 
$$
For the case $p=2$, we refer to \cite[page 51, Exercise 6]{Cox}.

Using Lemma~\ref{lem:Fn} we can determine explicitly when there are zero terms in the sequence $\{F_n\}$. 

\begin{lemma} \label{lem:zero}
Let $\lambda_1, \lambda_2$ be the two roots of $f(X)$ in $\overline{\F_q(x)}$. 
\begin{itemize}
\item[(a)] If $\Delta = 0$, then for any integer $n \ge 1$, $F_n = 0$ if and only if $p \mid n$. 

\item[(b)] If $\Delta \ne 0$, then  $F_n = 0$ for some positive integer $n$ if and only if $\dfrac{a^2}{b} \in \F_q$. 
Moreover, when $\dfrac{a^2}{b} \in \F_q$, we have that $\lambda_2/\lambda_1 \in \F_{q^2}$, 
and that for any integer $n \ge 1$, $F_n=0$ if and only if the order of $\lambda_2/\lambda_1$ divides $n$. 
\end{itemize}
\end{lemma}

\begin{proof}
Part (a) follows directly from Lemma~\ref{lem:Fn}, because $\Delta = 0$ and the characteristic of $\F_q$ is $p$. 

Part (b). Note that $\Delta \ne 0$. 
First, we suppose that $F_n = 0$ for some positive integer $n$. 
Then, by Lemma~\ref{lem:Fn}, we have $\lambda_1^n-\lambda_2^n = 0$, 
that is, $(\lambda_2 / \lambda_1)^n = 1$, which implies $\lambda_2 / \lambda_1 \in \overline{\F}_q$.  
Since the field extension degree $[\F_q(x, \lambda_1):\F_q(x)] \le 2$, we in fact have $\lambda_2 / \lambda_1 \in \F_{q^2}$.
Denote $\gamma = \lambda_2 / \lambda_1$. Then, $\gamma \in \F_{q^2}$ and $\lambda_2 = \gamma\lambda_1$. 
Since $\lambda_1, \lambda_2$ are the two roots of the equation $X^2 - aX -b = 0$, we have 
$$
\lambda_1 + \lambda_2 = a, \qquad \lambda_1 \lambda_2 = -b. 
$$
So, we obtain 
$$
(\gamma +1)\lambda_1 = a, \qquad \gamma\lambda_1^2 = -b, 
$$
which gives 
$$
\frac{a^2}{b} = - \frac{(\gamma+1)^2}{\gamma} \in \overline{\F}_q. 
$$
Hence, noticing $\dfrac{a^2}{b} \in \F_q(x)$, we must have $\dfrac{a^2}{b} \in \F_q$.  

Conversely, we assume $\dfrac{a^2}{b} \in \F_q$.  
Notice that 
$$
 \frac{\lambda_1}{\lambda_2} + \frac{\lambda_2}{\lambda_1} + 2 = \frac{(\lambda_1 + \lambda_2)^2}{\lambda_1\lambda_2} =
 \frac{a^2}{-b} \in \F_q.
$$ 
Then,  $\lambda_1/\lambda_2 + \lambda_2/\lambda_1 \in \F_q$. 
So, by valuation we must have $\lambda_2/\lambda_1 \in \overline{\F}_q$. 
Let $m$ be the order of $\lambda_2/\lambda_1$. Then, by Lemma~\ref{lem:Fn} we obtain $F_m = 0$. 

The rest of Part (b) is obvious. 
\end{proof}

Now, we use Lemmas~\ref{lem:Fn} and \ref{lem:zero} to determine under which condition $F_n = 0$ and $F_{n+1}=1$ for some positive integer $n$. 

\begin{lemma} \label{lem:zero-one}
We have that $F_n = 0$ and $F_{n+1}=1$ for some positive integer $n$ if and only if both $a$ and $b$ are in $\F_q$. 
Moreover, if both $a$ and $b$ are in $\F_q$ and let $\lambda_1, \lambda_2$ be the two roots of $f(X)$, then we have: 
\begin{itemize}
\item[(a)] if $\Delta = 0$, then for any integer $n \ge 1$, $F_n = 0$ and $F_{n+1} = 1$ if and only if $p \mid n$ and the order of $\lambda_1$ divides $n$; 

\item[(b)] if $\Delta \ne 0$, then for any integer $n \ge 1$, $F_n = 0$ and $F_{n+1} = 1$ if and only if 
the order of $\lambda_2/\lambda_1$ divides $n$ and the order of $\lambda_1$ divides $n$. 
\end{itemize}
\end{lemma}
\begin{proof}
Let $\lambda_1, \lambda_2$ be the two roots of $f(X)$. 

For sufficiency, we suppose that both $a$ and $b$ are in $\F_q$. 
Then, $f(X)$ is in fact a polynomial over $\F_q$, and so its roots $\lambda_1, \lambda_2$ are both in $\overline{\F_q}$. 
Hence, we can choose a positive integer $n$ such that $p \mid n$ and $\lambda_1^n = \lambda_2^n = 1$. 
By Lemma~\ref{lem:Fn}, it is easy to see that $F_n = 0$ and $F_{n+1} = 1$. 
This proves the sufficiency. 

For necessity, we suppose that $F_n = 0$ and $F_{n+1}=1$ for some positive integer $n$. 
If furthermore $\Delta = 0$, then by Lemma~\ref{lem:zero} we have $p \mid n$, and so by Lemma~\ref{lem:Fn} 
we have 
$$
F_{n+1} = (n+1)\lambda_1^n = \lambda_1^n,
$$
 which together with $F_{n+1}=1$ implies $\lambda_1^n = 1$, 
and thus $\lambda_1 \in \overline{\F_q}$, and then both $a$ and $b$ are in $\F_q$. 
Otherwise, if furthermore $\Delta \ne 0$, then by Lemma~\ref{lem:zero} we have $(\lambda_2/\lambda_1)^n = 1$ (that is, $\lambda_1^n = \lambda_2^n$), 
and so by Lemma~\ref{lem:Fn} we obtain 
$$
F_{n+1} = \frac{\lambda_1^{n+1}-\lambda_2^{n+1}}{\lambda_1-\lambda_2} 
= \frac{ \lambda_1^{n+1}- \lambda_1^{n} \lambda_2}{\lambda_1-\lambda_2} = \lambda_1^n,
$$
which together with $F_{n+1}=1$ implies $\lambda_1^n = 1$, 
and thus $\lambda_1 \in \overline{\F_q}$, and then both $a$ and $b$ are in $\F_q$. 
This completes the proof of the necessity. 

The second part of this lemma is in fact implied in the above arguments and Lemma~\ref{lem:zero}. 
\end{proof}

When the characteristic $p$ is odd, we get another formula for $F_n$ by using the discriminant $\Delta$. 

\begin{lemma} \label{lem:FnDelta}
    If the characteristic $p$ is odd, then for any integer $n \ge 1$ we have
    $$
    2^{n-1}F_n = \binom{n}{1} a^{n-1} + \binom{n}{3} a^{n-3}\Delta + \binom{n}{5} a^{n-5}\Delta^2 + \cdots; 
    $$
     if furthermore $n$ is odd, we have 
    $$
    2^{n-1}F_n = \binom{n}{1} a^{n-1} + \binom{n}{3} a^{n-3}\Delta+ \cdots + \binom{n}{n-2}a^2\Delta^{\frac{n-3}{2}}+\Delta^{\frac{n-1}{2}}.
    $$
\end{lemma}

\begin{proof}
    Suppose $\Delta \ne 0$. Then, the characteristic polynomial
    $f(X)$ has two distinct roots $\lambda_1,\lambda_2$, where (since $p$ is odd)
     $$\lambda_1=\dfrac{a+\sqrt{\Delta}}{2},\quad \lambda_2=\dfrac{a-\sqrt{\Delta}}{2}. $$
    Using Lemma~\ref{lem:Fn},  we obtain 
    \begin{align*}
        2^{n-1}F_n&=2^{n-1} \dfrac{\lambda_1^n-\lambda_2^n}{\lambda_1-\lambda_2}=\dfrac{(a+\sqrt{\Delta})^n-(a-\sqrt{\Delta})^n}{2\sqrt{\Delta}} \\
        &=\dfrac{2\binom{n}{1} a^{n-1}\sqrt{\Delta}+2\binom{n}{3} a^{n-3}(\sqrt{\Delta})^3 +2 \binom{n}{5} a^{n-5}(\sqrt{\Delta})^5+ \cdots}{2\sqrt{\Delta}} \\
        &= \binom{n}{1} a^{n-1}+ \binom{n}{3} a^{n-3}\Delta + \binom{n}{5} a^{n-5}\Delta^2+ \cdots.
    \end{align*}
    
    Now, we suppose $\Delta = 0$. 
    Then, $f(X)$ has a related root, say $\lambda_1$. 
    By Lemma~\ref{lem:Fn}, we have $F_n=n\lambda_1^{n-1}$. 
    Thus, noticing $2\lambda_1 = a$, we have 
    $$2^{n-1}F_n=n(2\lambda_1)^{n-1}=na^{n-1}.$$
\end{proof}

Finally, we present two lemmas about the relations among the terms in the sequence $\{F_n\}$. 

\begin{lemma} \label{lem:Fnn}
    For any integer $n \ge 1$, we have 
$$
F_n^2-F_{n+1}F_{n-1}=(-b)^{n-1},
$$
and 
$$(bF_{n-1})^2+bF_n(aF_{n-1}-F_n)=(-b)^n.$$
\end{lemma}

\begin{proof}
    If $\Delta \ne 0$, using Lemma~\ref{lem:Fn} 
    we have
    \begin{align*}
        F_n^2-F_{n+1}F_{n-1}&=\dfrac{(\lambda_1^n-\lambda_2^n)^2}{(\lambda_1-\lambda_2)^2}-\dfrac{(\lambda_1^{n+1}-\lambda_2^{n+1})(\lambda_1^{n-1}-\lambda_2^{n-1})}{(\lambda_1-\lambda_2)^2} \\
        &=\dfrac{(\lambda_1\lambda_2)^{n-1}(\lambda_1-\lambda_2)^2}{(\lambda_1-\lambda_2)^2}=(\lambda_1\lambda_2)^{n-1}=(-b)^{n-1}.
    \end{align*}

    If $\Delta = 0$,  using Lemma~\ref{lem:Fn} we obtain 
       $$F_n^2-F_{n+1}F_{n-1}=n^2\lambda_1^{2(n-1)}-(n+1)\lambda_1^n(n-1)\lambda_1^{n-2}=\lambda_1^{2(n-1)}=(-b)^{n-1}.$$

 Substituting $F_{n+1}=aF_n+bF_{n-1}$ into the first identity, we obtain the second identity.  
\end{proof}

\begin{lemma} \label{lem:Fnp}
    For any positive integer $n$, we have $F_{np}=(F_n)^pF_p$. 
\end{lemma}

\begin{proof}
If $\Delta \ne 0$, using Lemma~\ref{lem:Fn} we have
$$
F_{np}=\frac{\lambda_1^{np}-\lambda_2^{np}}{\lambda_1-\lambda_2}
=\frac{\lambda_1^{np}-\lambda_2^{np}}{\lambda_1^p - \lambda_2^p}\cdot\frac{\lambda_1^p-\lambda_2^p}{\lambda_1-\lambda_2}.
$$
Noticing $\lambda_1^{np}-\lambda_2^{np} = (\lambda_1^n-\lambda_2^n)^p$ and $\lambda_1^{p}-\lambda_2^{p} = (\lambda_1-\lambda_2)^p$, we get 
    $$
    F_{np}=\left(\frac{\lambda_1^{n}-\lambda_2^{n}}{\lambda_1 - \lambda_2}\right)^p \cdot \frac{\lambda_1^p-\lambda_2^p}{\lambda_1-\lambda_2}
          = (F_n)^pF_p.
    $$

    If $\Delta = 0$,  using Lemma~\ref{lem:Fn} we obtain  $F_p=p\lambda^{p-1}=0 ,F_{np}=np\lambda^{np-1} = 0$.
   So, we also have $F_{np}=(F_n)^p F_p$.
\end{proof}

\section{Computing $\alpha(M)$ and $\pi(M)$} 
\label{sec:M}
In this paper, we try to construct an algorithm for computing  $\alpha(M)$ and $\pi(M)$ when $a,b$ and $M$ are given. 
Clearly, by definition, for any non-zero element $c \in \F_q^*$ we have $\alpha(cM)=\alpha(M)$ and $\pi(cM)=\pi(M)$.  
As we all know, all polynomials in $\F_q[x]$ can be decomposed into the product of some irreducible polynomials. 
The strategy for computing $\alpha(M)$ and $\pi(M)$ is to reduce the problem to the special case when $M$ is a power of an irreducible polynomial. 

As usual, for any two integers $m_1, m_2$, let $[m_1, m_2]$ be the least common multiple of $m_1$ and $m_2$. 
Besides, for any two polynomials $M_1,M_2 \in \F_q[x]$, denote by $[M_1,M_2]$ the common multiple of $M_1$ and $M_2$ 
which is monic and of smallest degree.

\begin{theorem} \label{thm:lcm}
    For any non-constant polynomials $M_1,M_2 \in \F_q[x]$ which are both relatively prime to $b$, we have:  
\begin{itemize}    
   \item[(a)] $\alpha([M_1,M_2])=[\alpha(M_1),\alpha(M_2)]$;

    \item[(b)] $\pi([M_1,M_2])=[\pi(M_1),\pi(M_2)]$.
\end{itemize}
\end{theorem}

\begin{proof}
    Let $M=[M_1,M_2], \alpha=\alpha(M), \alpha_1=\alpha(M_1),\alpha_2=\alpha(M_2).$

    On the one hand, since $F_{\alpha} \equiv 0 \pmod{M}$, $M_1\mid M$ and $M_2 \mid M$,
    we have $F_{\alpha} \equiv 0\pmod{M_1}$ and $F_{\alpha} \equiv 0\pmod{M_2}$. 
    So, by \eqref{eq:Fn0-M}, $\alpha_1 \mid \alpha$ and $\alpha_2 \mid \alpha$, and then we obtain $[\alpha_1, \alpha_2] \mid \alpha$. 
    
    On the other hand, since $F_{\alpha_1} \equiv 0\pmod{M_1}$ and $F_{\alpha_2} \equiv 0\pmod{M_2}$,
    we have $F_{[\alpha_1,\alpha_2]} \equiv 0\pmod{M_1}$ and $F_{[\alpha_1,\alpha_2]} \equiv 0\pmod{M_2}$. 
    Thus, we get $F_{[\alpha_1,\alpha_2]} \equiv 0 \pmod{M}$, which implies  $\alpha \mid [\alpha_1,\alpha_2] $.

    Hence, we obtain  $\alpha = [\alpha_1,\alpha_2] $, as claimed in Part (a). 

Similarly, one can prove Part (b).    
\end{proof}

\begin{corollary}   \label{cor:M1M2}
If $M_1 \mid M_2$, then $\alpha(M_1) \mid \alpha(M_2)$ and $\pi(M_1) \mid \pi(M_2)$. 
\end{corollary}

Assume that $M$ has the prime decomposition $M=cP_1^{e_1}P_2^{e_2} \cdots P_k^{e_k}$ for some $c \in \F_q^*$ and some distinct monic irreducible polynomials $P_1, \ldots, P_k \in \F_q[x]$.
Then, by Theorem~\ref{thm:lcm}, we have:
\begin{align*}
& \alpha(M)=[\alpha(P_1^{e_1}),\alpha(P_2^{e_2}),...,\alpha(P_k^{e_k})], \\
& \pi(M)=[\pi(P_1^{e_1}),\pi(P_2^{e_2}),...,\pi(P_k^{e_k})].
\end{align*}
Hence, for computing $\alpha(M)$ and $\pi(M)$, we only need to compute $\alpha(P^e)$ and $\pi(P^e)$ for a power $P^e$ of an irreducible polynomial $P$, 
which will be treated in the next section.

\section{Computing $\alpha(P^e)$ and $\pi(P^e)$}
\label{sec:Pe}

\subsection{Outline}
In this section, we want to compute $\alpha(P^e)$ and $\pi(P^e)$, 
where $P$ is an irreducible polynomial in $\F_q[x]$ and $e$ is a positive integer. 
Roughly speaking, the strategy for computing $\alpha(P^e)$ and $\pi(P^e)$ is to reduce the problem to the special case when $e=1$. 

We remark that comparing this section with the related section in \cite{Ren} one can see that the polynomial case is quite different from the integer case 
and is much more complicated (see, for instance, Theorem~\ref{thm:ei'}). This can be imagined from the following lemma (compared to Proposition 1 of \cite{Ren}), which will be used again and again later on. 

Recall that the prime $p$ is the characteristic of $\F_q$. 

\begin{lemma} \label{lem:Pj} 
Let $P$ be an irreducible polynomial in $\F_q[x]$, and let $e$ be a positive integer. 
Then, for any integer $j$ with $e+1 \le j \le ep$, we have: 
\begin{itemize}        
   \item[(a)] $\alpha(P^j)=\alpha(P^e)$ or $p\alpha(P^e)$;
    
  \item[(b)] $\pi(P^j)=\pi(P^e)$ or $p\pi(P^e)$. 
\end{itemize}
\end{lemma}

\begin{proof}
   Part (a). By definition, we have $U^{\alpha(P^e)} \equiv sI \pmod{P^e}$ for some polynomial $s$. 
    Then, $U^{\alpha(P^e)}=sI+P^e B$ for some matrix $B$, 
    and so we have 
$$U^{p{\alpha(P^e)}}=(sI+P^e B)^{p} = s^pI + P^{ep}B^p.$$
    Thus, we get $U^{p{\alpha(P^e)}} \equiv s^p I \pmod{P^j}$ for any integer $j$ with $e+1 \le j \le ep$. 
    This, together with \eqref{eq:Uns}, gives $\alpha(P^j) \mid p\alpha(P^e)$. 
    Besides, by Corollary~\ref{cor:M1M2}, we have $\alpha(P^e) \mid \alpha(P^j)$. 
Hence, we obtain 
    $$\alpha(P^j)=\alpha(P^e) \text{ or } p \alpha(P^e).$$
This proves Part (a). 

   Part (b).  By applying similar arguments as the above and using \eqref{eq:UnI} and Corollary~\ref{cor:M1M2}, one can prove Part (b). 
\end{proof}

To make the presentations clear, in the following we compute $\alpha(P^e)$ and $\pi(P^e)$ separately.

\subsection{Computing $\alpha(P^e)$}  \label{sec:alpha}
In this section, we fix an irreducible polynomial $P \in \F_q[x]$ which is relatively prime to $b$.  
By Corollary~\ref{cor:M1M2}, the value $\alpha(P^e)$ increases when $e$ increases, 
but it might be not strictly increasing. 

We first handle two simple cases. 

\begin{theorem} \label{thm:alpha0}
Let $\lambda_1, \lambda_2$ be the two roots of the characteristic polynomial $f(X)$  in $\overline{\F_q(x)}$. 
Then, we have: 
\begin{itemize}
\item[(a)] if $\Delta = 0$, then for any integer $e \ge 1$, $\alpha(P^e)=p$;  

\item[(b)] if $\Delta \ne 0$ and $\dfrac{a^2}{b} \in \F_q$, let $m$ be the order of $\lambda_2/\lambda_1$, 
then for any integer $e \ge 1$, $\alpha(P^e)$ divides $m$, and moreover, $\alpha(P^e) = m$ for any sufficiently large $e$. 
\end{itemize}
\end{theorem}

\begin{proof}
Part (a). Since $\Delta = 0$, by Lemma~\ref{lem:zero} we have $F_p = 0$.
So, for any integer $e \ge 1$ we have $\alpha(P^e)=p$. 

Part (b). 
Since $\Delta \ne 0$ and $\dfrac{a^2}{b} \in \F_q$, by Lemma~\ref{lem:zero} we have $F_m = 0$. 
Then, using \eqref{eq:Fn0-M} we obtain that for any integer $e \ge 1$, $\alpha(P^e)$ divides $m$. 
Moreover, when 
$$
\deg(P^e) > \max\{\deg(F_1), \deg(F_2), \ldots, \deg(F_{m-1})\},
$$
 clearly we have $\alpha(P^e) = m$. 
\end{proof}

We remark that in Part (b) of Theorem~\ref{thm:alpha0}, since $\lambda_2/\lambda_1 \in \F_{q^2}$ (due to Lemma~\ref{lem:zero}), 
we have that $m$ divides $q^2-1$.  

Therefore, it remains to consider the case when $\Delta \ne 0$ and $\dfrac{a^2}{b} \not\in \F_q$. 
In this case, to compute $\alpha(P^e)$ for every integer $e \ge 1$, our approach is to determine each integer $e$ with $\alpha(P^e) \ne \alpha(P^{e+1})$. 

Let $e_1$ be the first positive integer $e$ such that $\alpha(P^e) \neq \alpha(P^{e+1})$. 
More general, for any integer $i \ge 1$,  let $e_i$ be the $i$-th positive integer $e$ such that $\alpha(P^e) \neq \alpha(P^{e+1})$ 
(that is, $\alpha(P^{e+1})=p\alpha(P^{e})$ by Lemma~\ref{lem:Pj}). 

The following lemma ensures the existence of the above integers $e_i$. 

\begin{lemma}  \label{lem:ei}
Assume $\Delta \ne 0$ and $\dfrac{a^2}{b} \not\in \F_q$. 
Then, for any integer $i \ge 1$,  the integer $e_i$ always exists. 
\end{lemma}
\begin{proof}
Since $\Delta \ne 0$ and $\dfrac{a^2}{b} \not\in \F_q$, by Lemma~\ref{lem:zero} we have that $F_n \ne 0$ for any integer $n \ge 1$. 
We also notice that $F_0 = 0$, and the sequence $\{F_n\}$ is a periodic sequence modulo $P^k$ for any integer $k \ge 1$ (due to $\gcd(b, P)=1$). 
Then, for any integer $k \ge 1$, there is a positive integer $n_k$ such that $F_{n_k} \ne 0$ and $F_{n_k} \equiv 0 \pmod{P^k}$. 
This implies that $e_i$ always exists for any integer $i \ge 1$. 
\end{proof}

Clearly, for any integer $i \ge 1$, by the choice of $e_i$, we know that 
\begin{equation}  \label{eq:FPei}
F_{\alpha(P^{e_i})} \equiv 0 \pmod{P^{e_i}}   \quad \text{and} \quad F_{\alpha(P^{e_i})} \not\equiv 0 \pmod{P^{e_i +1}}. 
\end{equation}

The following theorem enables us to determine each $e_i$ if $e_1$ is given, 
and it also implies that the sequence $\{e_{i+1} - e_i\}$ is a geometric sequence with common ratio $p$. 

\begin{theorem} \label{thm:ei}
Assume $\Delta \ne 0$ and $\dfrac{a^2}{b} \not\in \F_q$. 
For every integer $i \ge 1$, we have: 
\begin{itemize}
\item[(a)] if $\Delta \equiv 0 \pmod{P}$, then $e_{i} = \dfrac{(p^i -1)e_1}{p-1}$, and $\alpha(P^{e_i})=p^i$; 

\item[(b)] if $\Delta \not\equiv 0 \pmod{P}$, then $e_{i} = p^{i-1}e_{1}$, and $\alpha(P^{e_i})=p^{i-1}\alpha(P)$.
\end{itemize}
\end{theorem}

\begin{proof}
    Part (a). 
    Let $d = \deg(P)$. Then, it is well-known that the residue calss ring $\F_q[x]/P\F_q[x]$ is in fact the finite field $\F_{q^d}$ of $q^d$ elements. 
    Since $\Delta = a^2+4b \equiv 0 \pmod{P}$, 
    we know that $f(X)$ has a repeated root modulo $P$ (that is, over the field $\F_{q^d}$), say $\lambda$.
    Moreover, due to $b \not\equiv 0 \pmod{P}$, we have   $\lambda \not\equiv 0$ modulo $P$. 
    As in Lemma~\ref{lem:Fn} and noticing that the residue class ring $\F_q[x]/P\F_q[x]$ is in fact a field, we obtain that $F_n \equiv n\lambda^{n-1} \pmod{P}$ for any integer $n \ge 0$. 
    So, we directly have $F_p \equiv 0 \pmod{P}$, and $F_{n} \not\equiv 0 \pmod{P}$ for any $n$ with $1 \le n <p$. 
Thus, we get $\alpha(P) = p$. 

By the choice of $e_1$, we know that $\alpha(P^{e_1}) = \alpha(P) = p$ and $\alpha(P^{e_1 +1}) =p \alpha(P^{e_1})=p^2$. 
So, we have 
\begin{equation}  \label{eq:Fp}
F_p \equiv 0 \pmod{P^{e_1}}   \quad \text{and} \quad F_p \not\equiv 0 \pmod{P^{e_1 +1}}. 
\end{equation}

Now, for any integer $i \ge 1$, by the choice of $e_{i+1}$ (it indeed exists by Lemma~\ref{lem:ei}), we have 
\begin{equation*}
\alpha(P^{e_{i+1}}) = \alpha(P^{e_i +1}) =p\alpha(P^{e_i}).
\end{equation*} 
Then, using Lemma~\ref{lem:Fnp}, \eqref{eq:FPei} and \eqref{eq:Fp}, we obtain  
\begin{equation*}
F_{\alpha(P^{e_{i+1}})} = F_{p\alpha(P^{e_i})} = (F_{\alpha(P^{e_i})})^p F_p  \equiv 0 \pmod{P^{pe_i + e_1}}, 
\end{equation*}
and 
\begin{equation*}
F_{\alpha(P^{e_{i+1}})} = (F_{\alpha(P^{e_i})})^p F_p \not\equiv 0 \pmod{P^{pe_i + e_1 +1}}.
\end{equation*}
Hence, we have 
\begin{equation}   \label{eq:eei1}
e_{i+1}=pe_i+e_1. 
\end{equation}
Therefore, using \eqref{eq:eei1} repeatedly, we get that for any integer $i \ge 1$, 
\begin{equation*}
e_{i}=(p^{i-1} + \cdots + p + 1)e_1 = \frac{(p^i -1)e_1}{p-1}, 
\end{equation*}
and by definition, $\alpha(P^{e_i})=p^i$ (since $\alpha(P)=p$). 
This completes the proof of Part (a).

    Part (b). 
Since $\Delta \not\equiv 0 \pmod{P}$, 
    the polynomial $f(X)$ has two distinct root modulo $P$, say $\lambda_1$ and $\lambda_2$. 
    As in Lemma~\ref{lem:Fn}, we obtain $F_n \equiv \dfrac{\lambda_1^n-\lambda_2^n}{\lambda_1-\lambda_2} \pmod{P}$. 
Then, we get that $\alpha(P)$ is equal to the least positive integer $n$ such that $\lambda_1^n - \lambda_2^n \equiv 0 \pmod{P}$. 
Since $\lambda_1$ and $\lambda_2$ are distinct modulo $P$, we have $\lambda_1^p -   \lambda_2^p = (\lambda_1 -   \lambda_2)^p \not\equiv 0 \pmod{P}$. 
So, we obtain 
\begin{equation}  \label{eq:Fp0}
F_p \not\equiv 0 \pmod{P}.
\end{equation}

Now, as in proving Part (a) and using Lemma~\ref{lem:Fnp}, \eqref{eq:FPei} and \eqref{eq:Fp0}, we get that for any integer $i \ge 1$, 
\begin{equation}   \label{eq:eei2}
e_{i+1}=pe_i. 
\end{equation}
Thus, we obtain 
\begin{equation*}
e_{i}=p^{i-1} e_1  
\quad  \text{and} \quad
    \alpha(P^{e_i})=p^{i-1} \alpha(P).
\end{equation*}
This completes the proof of Part (b). 
\end{proof}

In Theorem~\ref{thm:ei}, for computing the integer $e_1$, we first find the smallest positive integer $n$ such that $F_n \equiv 0 \pmod{P}$, 
and then $e_1$ is exactly the exponent of $P$ in the prime factorization of $F_n$. 

\begin{remark}  \label{rem:ei}
In the proof of Theorem~\ref{thm:ei}, we only use the condition ``$\Delta \ne 0$ and $\dfrac{a^2}{b} \not\in \F_q$" 
to ensure the existence of the integer $e_i$ for any integer $i \ge 1$.  
So, without this condition, for any integer $i \ge 1$ if the integer $e_i$ exists, then the integer $e_j$ exists for any $1 \le j \le i$, 
and thus $e_i$ can be computed following the two formulas in Theorem~\ref{thm:ei}. 
\end{remark}

Now, we first compute $e_1$ by direct computation, and then using Theorem~\ref{thm:ei} we can compute $e_i$ for any $i \ge 2$. 
Then, noticing the choices of the integers $e_i$,  
we directly obtain the following result about computing the value $\alpha(P^e)$. 

\begin{theorem}  \label{thm:alpha}
Assume $\Delta \ne 0$ and $\dfrac{a^2}{b} \not\in \F_q$. Put $e_0 = 0$. 
Then, for any integer $e \ge 1$, we have that  $e_{i-1} < e \le e_i$ for some integer $i \ge 1$, and moreover, 
\begin{itemize}
\item[(a)] if $\Delta \equiv 0 \pmod{P}$, then  $\alpha(P^e) = \alpha(P^{e_i})=p^i$; 

\item[(b)] if $\Delta \not\equiv 0 \pmod{P}$, then  $\alpha(P^e) = \alpha(P^{e_i})=p^{i-1}\alpha(P)$.
\end{itemize}
\end{theorem}

\subsection{Computing $\pi(P^e)$}  \label{sec:pi}
In this section, we still fix an irreducible polynomial $P \in \F_q[x]$ which is relatively prime to $b$.  
We want to compute the value $\pi(P^e)$ for any integer $e \ge 1$.

We first handle the simple case when both $a$ and $b$ are in $\F_q$. 

\begin{theorem} \label{thm:pi0}
Assume that both $a$ and $b$ are in $\F_q$. 
Let $\lambda_1, \lambda_2$ be the two roots of the characteristic polynomial $f(X)$ in $\overline{\F}_q$. 
Then, we have: 
\begin{itemize}
\item[(a)] if $\Delta = 0$, let $m_1$ be the product of $p$ and the order of $\lambda_1$, then for any integer $e \ge 1$, $\pi(P^e)$ divides $m_1$, 
and moreover, $\pi(P^e) = m_1$ for any sufficiently large $e$;  

\item[(b)] if $\Delta \ne 0$, let $m_2$ be the least common multiple of 
the order of $\lambda_2/\lambda_1$ and the order of $\lambda_1$, 
then for any integer $e \ge 1$, $\alpha(P^e)$ divides $m_2$, and moreover, $\pi(P^e) = m_2$ for any sufficiently large $e$. 
\end{itemize}
\end{theorem}
\begin{proof}
Part (a). If $\Delta = 0$, then using Part (a) of Lemma~\ref{lem:zero-one} and noticing that $p$ and the order of $\lambda_1$ are coprime, 
we have that $F_{m_1} = 0$ and $F_{m_1+1} = 1$. 
Then, using \eqref{eq:Fn01-M} we obtain that for any integer $e \ge 1$, $\pi(P^e)$ divides $m_1$. 
Moreover, when 
$$
\deg(P^e) > \max\{\deg(F_1), \deg(F_2), \ldots, \deg(F_{m_1-1})\},
$$
 clearly we have $\pi(P^e) = m_1$.

Part (b). Similarly, one can prove Part (b) by using Part (b) of Lemma~\ref{lem:zero-one} and \eqref{eq:Fn01-M}. 
Moreover, when 
$$
\deg(P^e) > \max\{\deg(F_1), \deg(F_2), \ldots, \deg(F_{m_2-1})\},
$$
 we have $\pi(P^e) = m_2$. 
\end{proof}

We remark that in Theorem~\ref{thm:pi0}, since $\lambda_1, \lambda_2 \in \F_{q^2}$ (due to $a, b \in \F_q$),  
we have that $m_1$ divides $p(q^2-1)$, and $m_2$ divides $q^2-1$.  

Hence, it remains to consider the case when $a$ and $b$ are not both in $\F_q$. 
Our approach is to determine each integer $e$ with $\pi(P^e) \ne \pi(P^{e+1})$. 

Similarly, for any integer $i \ge 1$ we define $e_i^{\prime}$ to be the $i$-th positive integer $e \text{ such that } \pi(P^e) \neq \pi(P^{e+1})$.    

The following lemma ensures the existence of the above integers $e_i^{\prime}$. 

\begin{lemma}   \label{lem:ei'}
Assume that $a$ and $b$ are not both in $\F_q$. 
Then, for any integer $i \ge 1$,  the integer $e_i^{\prime}$ always exists. 
\end{lemma}
\begin{proof}
Since $a$ and $b$ are not both in $\F_q$, 
by Lemma~\ref{lem:zero-one} we obtain that there is no positive integer $n$ such that $F_n = 0$ and $F_{n+1}=1$.
We also notice that $F_0 = 0, F_1=1$, and the sequence $\{F_n\}$ is a periodic sequence modulo $P^k$ for any integer $k \ge 1$ (due to $\gcd(b, P)=1$). 
Then, for any integer $k \ge 1$, there is a positive integer $n_k$ such that $F_{n_k} \equiv 0, F_{n_k +1} \equiv 1 \pmod{P^k}$. 
This implies that $e_i^{\prime}$ always exists for any integer $i \ge 1$. 
\end{proof}

Recalling the integers $e_i$ defined in Section~\ref{sec:alpha}, we establish the following relation between $e_i$ and $e_i^{\prime}$ for further deductions. 

\begin{lemma}  \label{lem:eiei'}
For any integer $i \ge 1$, if the two integers $e_i, e_i^{\prime}$ both exist, then we have  $e_i^{\prime} \leq e_i$.     
\end{lemma}

\begin{proof}
    Since the integer $e_i$ exists, we have that the integer $e_1$ also exists. 
    By the choice of the integer $e_1$, we have $\alpha(P)=\cdots=\alpha(P^{e_1})\neq\alpha(P^{e_1+1})$.  
    Let $n = \alpha(P)$. 
    Recall the matrix $U$ defined at the beginning of Section~\ref{sec:pre}. 
    Then, due to the choice of $n$, we have 
    $$
    U^n=sI+P^{e_1}B,
    $$ 
    where $s \in \F_q[x]$ with $\gcd(s,P)=1$, and $B$ is not a scalar matrix modulo $P$.
    
    Since the integer $e_i$ exists, noticing the choice of $e_i$ we obtain that $\alpha(P^{e_i}) =  p^{i-1} \alpha(P^{e_1}) =  p^{i-1} \alpha(P) = np^{i-1}$ and 
    $$
    U^{\alpha(P^{e_i})} = U^{np^{i-1}} \not\equiv \textrm{a scalar matrix} \pmod{P^{e_i +1}}, 
    $$
    which becomes 
    \begin{align*}
       U^{np^{i-1}} = (sI+P^{e_1}B)^{p^{i-1}} & =  s^{p^{i-1}}I+P^{p^{i-1}e_1}B^{p^{i-1}} \\
       & \not\equiv \textrm{a scalar matrix} \pmod{P^{e_i +1}}. 
    \end{align*}
    So, we have 
    \begin{equation}  \label{eq:PBe_i}
    P^{p^{i-1}e_1}B^{p^{i-1}} \not\equiv \textrm{a scalar matrix} \pmod{P^{e_i +1}}. 
    \end{equation}
    In addition, by Theorem~\ref{thm:ei} we have 
   $$
       e_i \le \frac{(p^i -1)e_1}{p-1} < 2p^{i-1}e_1, 
   $$
   which gives
   \begin{equation}  \label{eq:eie1}
      e_i + 1 \le 2p^{i-1} e_1. 
   \end{equation}
    
    Since the integer $e_i^{\prime}$ exists, so does the integer $e_1^{\prime}$.     Denote $k= \ord_{P}(s)$. 
    Since $\F_q[x]/P\F_q[x]$ is a finite field of characteristic $p$, we know that $p \nmid k$. 
   Noticing the choice of $e_i^{\prime}$ and using \eqref{eq:betas}, we obtain 
    $$
    \pi(P^{e_i^{\prime}}) =  p^{i-1} \pi(P^{e_1^{\prime}}) = p^{i-1} \pi(P) = p^{i-1}\alpha(P) \ord_P(s) = knp^{i-1}. 
    $$
    Then, we deduce that 
    \begin{equation}  \label{eq:Upiei}
      \begin{split}
         & U^{\pi(P^{e_i^{\prime}})}   = U^{knp^{i-1}}  = (s^{p^{i-1}}I+P^{p^{i-1}e_1}B^{p^{i-1}})^k \\ 
         & \quad \equiv s^{p^{i-1}k}I + ks^{p^{i-1}(k-1)}P^{p^{i-1}e_1}B^{p^{i-1}}  \quad \textrm{(due to \eqref{eq:eie1})}\\
         & \quad \not\equiv \textrm{a scalar matrix} \pmod{P^{e_i +1}} \quad \textrm{(due to \eqref{eq:PBe_i}),}
     \end{split}    
    \end{equation}
  where in the non-equivalence we also use the facts: $p \nmid k$ and $P \nmid s$.  
  On the other hand, by definition we directly have 
  $$
  U^{\pi(P^{e_i^{\prime}})} \equiv I \pmod{P^{e_i^{\prime}}}, 
  $$
  which, together with \eqref{eq:Upiei}, implies that 
  $$
  e_i^{\prime} \le e_i. 
  $$
\end{proof}

\begin{remark}  
\label{rem:eiei'}
Lemma~\ref{lem:eiei'} tells us that $e_i / e_i^{\prime} \ge 1$ whenever they both exist. 
In fact, $e_1 / e_1^{\prime}$ can be arbitrary large, and so is $e_i / e_i^{\prime}$. 
For example, choose $p=3$, $a=x+1$, $P = x+2$, and $b=2x^2+x+2+P^j$ for any fixed integer $j > 3$, 
then compute the first five terms: $F_0=0$, $F_1=1$, $F_2 = x+1$, $F_3 = P^j$, $F_4 = 2x^3+2 + 2(x+1)P^j \equiv 1 \pmod{P^3}$, 
and so we have $\alpha(P)=3$ and $e_1=j$, and $\pi(P)=3$ and $e_1^{\prime} = 3$. 
Hence, in this example, $e_1/e_1^{\prime} = j/3$, which tends to the infinity as $j$ goes to the infinity. 
\end{remark}

We also need the following lemma. 

\begin{lemma}
\label{lem:ei'p}
For any integer $i \ge 1$, if the three integers $e_i^{\prime}, e_{i+1}^{\prime}, e_{i+2}^{\prime}$ all exist 
and moreover $e_{i+1}^{\prime} = pe_{i}^{\prime}$, then we have  $e_{i+2}^{\prime} = pe_{i+1}^{\prime}$.     
\end{lemma}
\begin{proof}
First, by the choice of $e_i^{\prime}$,  we can write
$$
U^{\pi(P^{e_i^{\prime}})} = I + P^{e_i^{\prime}}B
$$
for some matrix $B$ satisfying $B \not\equiv 0 \pmod{P}$. 
Then, using the assumption $e_{i+1}^{\prime} = pe_{i}^{\prime}$, we have 
$$
U^{\pi(P^{e_{i+1}^{\prime}})} = U^{p\pi(P^{e_i^{\prime}})} = (I + P^{e_i^{\prime}}B)^p = I + P^{pe_i^{\prime}}B^p =  I + P^{e_{i+1}^{\prime}}B^p.
$$
By the choice of $e_{i+1}^{\prime}$, this implies that 
\begin{equation}
\label{eq:Bp0}
B^p \not\equiv 0 \pmod{P}.
\end{equation}

Now, let $d = \deg(P)$. We know that $\F_q[x]/P\F_q[x]$ is the finite field $\F_{q^d}$. 
Let $\lambda_1, \lambda_2$ be the two eigenvalues of the matrix $B \pmod{P}$ (view it as a matrix over $\F_{q^d}$). 
That is,  $\lambda_1, \lambda_2$ are some elements in some extension of $\F_{q^d}$ (they might be the same). 
Note that for some invertible matrix $T$, $T^{-1}BT \pmod{P}$ is either a diagonal form or a Jordan form. 
Then,  we have 
$$
B^p  \equiv T \begin{bmatrix}
        \lambda_1^p & 0 \\
        0 & \lambda_2^p
    \end{bmatrix} T^{-1}  \pmod{P}. 
$$
From \eqref{eq:Bp0}, we have either $\lambda_1 \ne 0$ or $\lambda_2 \ne 0$. 
So, we get 
$$
B^{p^2}  \equiv T \begin{bmatrix}
        \lambda_1^{p^2} & 0 \\
        0 & \lambda_2^{p^2}
    \end{bmatrix} T^{-1} \not\equiv 0  \pmod{P}. 
$$
That is, $B^{p^2} \not\equiv 0 \pmod{P}$. 
Hence, noticing 
$$
U^{\pi(P^{e_{i+2}^{\prime}})} = U^{p\pi(P^{e_{i+1}^{\prime}})} = (I + P^{e_{i+1}^{\prime}}B^p)^p = I + P^{pe_{i+1}^{\prime}}B^{p^2}, 
$$
we obtain 
$$
e_{i+2}^{\prime} = pe_{i+1}^{\prime}. 
$$
This completes the proof. 
\end{proof}

Now, we are ready to determine each integer $e_i^{\prime}$.   
Recall that $p$ is the characteristic of $\F_q$. 
For any non-zero polynomial $g(x) \in \F_q[x]$, let $v_P(g)$ be the maximal integer $n$ such that $P^n \mid g$. 
As usual, put $v_P(0) = + \infty$. 

\begin{theorem} \label{thm:ei'}
    Assume that $a$ and $b$ are not both in $\F_q$. 
    Then,  we have:
\begin{itemize}
    \item[(a)] if the integer $e_1$ does not exist (including the case $\Delta = 0$), then  for any integer $i \ge 1$, $e_i^{\prime} = p^{i-1}e_1^{\prime}$;

    \item[(b)] if $e_1^{\prime} < e_1$, then  for any integer $i \ge 1$, $e_i^{\prime} = p^{i-1}e_1^{\prime}$;  
        
    \item[(c)] if $\Delta \not\equiv 0 \pmod{P}$, then  for any integer $i \ge 1$, $e_i^{\prime} = p^{i-1}e_1^{\prime}$;
    
    \item[(d)] if $p>2$,  $\Delta \equiv 0 \pmod{P}$ and $e_1^{\prime} = e_1$,  let $k=\ord_P(a/2)$ and $m = v_P((a/2)^k -1)$, then for any integer $i \ge 2$, 
           \begin{equation*}
                  e_{i}^{\prime} = \min\left\{ mp^i,   \frac{(p^i - 1)e_1}{p-1} \right\};
          \end{equation*}    
    
    \item[(e)] if $p=2$,  $\Delta \equiv 0 \pmod{P}$ and $e_1^{\prime} = e_1$, let $k=\ord_P(b)$ and $m = v_P(b^k -1)$, 
      then we have $m \ge e_1$; and moreover,  if $m > e_1$, then for any $i \ge 2$,  
      $$
         e_{i}^{\prime} =    2^{i-1}e_1; 
      $$
      otherwise if $m = e_1$, let $c_1 = b^k - 1$, then for any $i \ge 2$,  
      $$
         e_{i}^{\prime} = \min\left\{(2^i - 1)e_1, \ v_P(c_i) \right\}, 
      $$
      where $c_i = a^{2^i - 2} + bc_{i-1}^2$. 
\end{itemize}
\end{theorem}

\begin{proof}
   First, since $a$ and $b$ are not both in $\F_q$, by Lemma~\ref{lem:ei'} we know that for any integer $i \ge 1$, the integer $e_i^{\prime}$ exists. 
   
   Part (a).  
   Let  $n = \alpha(P)$. 
   Since the integer $e_1$ does not exist (by Part (a) of Theorem~\ref{thm:alpha0}, the case $\Delta = 0$ satisfies this condition), 
   we have that $\alpha(P^e) = n$ for any integer $e \ge 1$.    
   Then, by definition, $U^n$ is a scalar matrix modulo $P^e$ for any integer $e \ge 1$. 
   So, $U^n$ must be a scalar matrix itself. 
   
   Assume that $U^n = sI$ for some polynomial $s \in \F_q[x]$. 
    Then, by \eqref{eq:betas}  we have that for any integer $e \ge 1$, 
    \begin{equation*}
       \pi(P^e) = \alpha(P^e)\ord_{P^e}(s) = n\ord_{P^e}(s). 
    \end{equation*}
    So, for any integer $i \ge 1$ we obtain   $\pi(P^{e_i^{\prime}}) = n \ord_{P^{e_i^{\prime}}}(s)$. 
    On the other hand, by definition we have $\pi(P^{e_i^{\prime}}) = p^{i-1} \pi(P^{e_1^{\prime}}) = p^{i-1}n \ord_{P^{e_1^{\prime}}}(s) = knp^{i-1}$, 
    where $k=\ord_{P^{e_1^{\prime}}}(s)$.  
    Hence, we obtain 
    \begin{equation}  \label{eq:ord-s}
           \ord_{P^{e_i^{\prime}}}(s) = kp^{i-1}. 
    \end{equation}
    In addition, due to $k = \ord_{P^{e_1^{\prime}}}(s)$, we obtain that for any integer $i \ge 1$, 
    \begin{equation}  \label{eq:skpi}
       s^{kp^{i-1}} \equiv 1 \pmod{P^{p^{i-1}e_1^{\prime}}}   \quad \textrm{and} \quad  s^{kp^{i-1}} \not\equiv 1 \pmod{P^{p^{i-1}e_1^{\prime}+1}}.
    \end{equation}
    This, together with \eqref{eq:ord-s} and the choice of $e_i^{\prime}$, gives 
    $$
       e_i^{\prime} = p^{i-1}e_1^{\prime}. 
    $$
    This completes the proof of Part (a).

    Part (b). 
    By the choice of the integer $e_1$, we have $\alpha(P)=\cdots=\alpha(P^{e_1})\neq\alpha(P^{e_1+1})$.  
    Let $n = \alpha(P)$. 
    Then, due to the choice of $n$, we have 
    \begin{equation}  \label{eq:Un}
           U^n=sI+P^{e_1}B,
    \end{equation}
    where $s \in \F_q[x]$ with $k=\ord_{P^{e_1^{\prime}}}(s)$,   and $B$ is not a scalar matrix modulo $P$.
    
    For any integer $i \ge 1$, by the choice of the integer $e_i^{\prime}$ and using $e_1^{\prime} < e_1$, we have 
    $$
      \pi(P^{e_i^{\prime}}) = p^{i-1}\pi(P^{e_1^{\prime}}) = p^{i-1}\alpha(P^{e_1^{\prime}}) \ord_{P^{e_1^{\prime}}}(s) = knp^{i-1}.
    $$ 
   Since $e_1^{\prime} < e_1$, we have $p^{i-1}e_1^{\prime} + 1 \le p^{i-1}e_1$. 
   Then, combining this with \eqref{eq:skpi}, we obtain 
   \begin{equation}  \label{eq:Upi}
   \begin{split}
         U^{\pi(P^{e_i^{\prime}})} = U^{knp^{i-1}} & = (sI + P^{e_1}B)^{kp^{i-1}} = (s^{p^{i-1}}I + P^{p^{i-1}e_1}B^{p^{i-1}})^{k} \\
         & \equiv s^{kp^{i-1}}I \equiv I  \pmod{P^{p^{i-1}e_1^{\prime}}}, 
   \end{split}
   \end{equation}
   and
   \begin{equation}  \label{eq:Upi+}
         U^{\pi(P^{e_i^{\prime}})}  \equiv s^{kp^{i-1}}I \not\equiv I  \pmod{P^{p^{i-1}e_1^{\prime} + 1}}. 
   \end{equation}
   From \eqref{eq:Upi} and \eqref{eq:Upi+} and noticing the choice of the integer $e_i^{\prime}$, we get 
   $$
       e_i^{\prime} = p^{i-1}e_1^{\prime}. 
    $$
    This completes the proof of Part (b).

   Part (c).  
   By Part (a) and Part (b) and using Lemma~\ref{lem:eiei'}, we only need to consider the case when $\Delta \not\equiv 0 \pmod{P}$ and $e_1^{\prime} = e_1$. 
   We prove the desired result of this case by considering two subcases: Case (c1) and Case (c2). 
   In Case (c1), we assume $\dfrac{a^2}{b} \not\in \F_q$. 
   In Case (c2), we assume $\dfrac{a^2}{b} \in \F_q$. 
   
   Case (c1): $\dfrac{a^2}{b} \not\in \F_q$. In this case, by Lemma~\ref{lem:ei} we know that for any integer $i \ge 1$, the integer $e_i$ exists; 
   and by Part (b) of Theorem~\ref{thm:ei} we have $e_i = p^{i-1}e_1$. 
   For any integer $i \ge 1$, by definition we have $U^{\pi(P^{e_i^{\prime}})} \equiv I \pmod{P^{e_i^{\prime}}}$. 
   Then, we have 
   $$
      U^{\pi(P^{e_{i+1}^{\prime}})} = U^{p\pi(P^{e_{i}^{\prime}})}\equiv I \pmod{P^{pe_i^{\prime}}}, 
   $$
   which implies that 
   \begin{equation}    \label{eq:eei'+}
      e_{i+1}^{\prime} \ge p e_{i}^{\prime}.
   \end{equation}
   
   Noticing $e_1^{\prime} = e_1$ and using Lemma~\ref{lem:eiei'} and \eqref{eq:eei'+}, we obtain 
   $$
     pe_1 = p e_{1}^{\prime} \le e_{2}^{\prime} \le e_2 = pe_1. 
   $$
   So, we have $e_{2}^{\prime} = p e_{1}^{\prime} = e_2$.    
   Similarly,  step by step we get that for any integer $i \ge 1$, 
   $$
       e_i^{\prime} = p^{i-1}e_1^{\prime} = e_i. 
    $$
    This completes the proof of Case (c1).

   Case (c2): $\dfrac{a^2}{b} \in \F_q$. In this case, by Part (b) of Theorem~\ref{thm:alpha0} and noticing the assumption on $e_1$, 
   we know that there exists a positive integer, say $j$, such that the integers $e_1, \ldots, e_j$ exist, but the integers $e_{j+1}, e_{j+2}, \ldots$ do not exist. 
   
   By Remark~\ref{rem:ei}, we have that for any integer $i$ with $1 \le i \le j$, $e_i = p^{i-1}e_1$.  
   Then, as in the proof of Case (c1) we obtain that 
   \begin{equation}   \label{eq:eij1}
            e_i^{\prime} = p^{i-1}e_1^{\prime}, \quad 1 \le i \le j. 
   \end{equation}
   Due to the choice of $j$ and let $n = \alpha(P)$, we know that for any integer $e > e_j$, 
   $$
      \alpha(P^e) = \alpha(P^{e_j+1}) = p\alpha(P^{e_j}) = p^j \alpha(P) = np^j. 
   $$
   This implies that $U^{np^j}$ is a scalar matrix modulo $P^e$ for any $e > e_j$. 
   So,  $U^{np^j}$ must be a scalar matrix. 
   Then, as in the proof of Part (a) we get that 
   \begin{equation}  \label{eq:eij2}
            e_i^{\prime} = p^{i-j}e_{j}^{\prime}, \quad i \ge  j. 
   \end{equation}
   Combining \eqref{eq:eij1} with \eqref{eq:eij2}, we obtain that   for any integer $i \ge 1$, 
   $$
       e_i^{\prime} = p^{i-1}e_1^{\prime}. 
    $$ 
    This completes the proof of Case (c2) and also the proof of Part (c).

   Part (d). 
  The existence of $e_1$ implies $\Delta \ne 0$ (noticing Part (a) of Theorem~\ref{thm:alpha0}).  
  If   $\dfrac{a^2}{b} \in \F_q$, then $a^2 = cb$ for some $c \in \F_q$, and so $\Delta = a^2 + 4b = (c+4)b$ with $c+4 \in \F_q$, 
  which, together with $\gcd(b,P)=1$ and $\Delta \equiv 0 \pmod{P}$, implies that $c+4 = 0$, and thus $\Delta = 0$, contradicting with $\Delta \ne 0$. 
  So, we must have $\dfrac{a^2}{b} \not\in \F_q$. 
  Hence, by Lemma~\ref{lem:ei} we know that for any integer $i \ge 1$, the integer $e_i$ exists. 
  
  Since $\Delta = a^2 + 4b \equiv 0 \pmod{P}$ and $p$ is odd,  the characteristic polynomial $f(X) = X^2 - aX - b$ has the repeated root $a/2$ modulo $P$. 
   Then, we can put the matrix $U$ into Jordan form modulo $P$: $U \equiv TJT^{-1} \pmod{P}$ for some invertible matrix $T$ and  
   $J =
    \begin{bmatrix}
        a/2 & 1 \\
        0 & a/2
    \end{bmatrix}$.
    Noticing $J^p = (a/2)^p I$, we have 
    \begin{equation}  \label{eq:Up}
       U^p \equiv (TJT^{-1})^p \equiv TJ^pT^{-1} \equiv (a/2)^p I  \pmod{P}. 
    \end{equation}
    Due to $\Delta \equiv 0 \pmod{P}$, by Part (a) of Theorem~\ref{thm:ei} we have $\alpha(P) = p$. 
    So, we have that $F_p \equiv 0 \pmod{P^{e_1}}$ and $\gcd(a/2, P)=1$. 
    Recall that we have assumed $k = \ord_P(a/2)$. 
    Then, we obtain that $P \mid (a/2)^k - 1$ and 
    $$ 
       \ord_P((a/2)^p) = \ord_P(a/2) = k.
    $$  
    Combining this with \eqref{eq:betas} and \eqref{eq:Up}, we get 
    $$
       \pi(P) = p\ord_P((a/2)^p) = kp. 
    $$
    
    By Lemma \ref{lem:FnDelta} and noticing $F_p \equiv 0 \pmod{P^{e_1}}$,  
    we have that $2^{p-1}F_p =\Delta^{(p-1)/2}\equiv 0\pmod{P^{e_1}}$,
    that is, $a^2+4b = \Delta \equiv 0 \pmod{P^{\lceil 2e_1/(p-1) \rceil}}$, where $\lceil 2e_1/(p-1) \rceil$ is the smallest integer not less than $2e_1/(p-1)$. 
    Thus, there exists a polynomial $h \in \F_q[x]$ such that $a^2+4b = h P^{\lceil 2e_1/(p-1) \rceil}$, which gives 
    $$
      -b=\frac{a^2}{4}-\frac{h}{4}P^{\lceil 2e_1/(p-1) \rceil}.
    $$
    
    For any integer $i \ge 1$, by Part (a) of Theorem~\ref{thm:ei}, we have  $e_i=\dfrac{(p^i-1)e_1}{p-1}$ and $\alpha(P^{e_i}) = p^{i}$. 
    Clearly, we have 
    \begin{equation}  \label{eq:piei}
       p^i \lceil 2e_1/(p-1) \rceil \ge p^i \cdot 2e_1/(p-1) > e_i. 
    \end{equation}   
    Then, using Lemma~\ref{lem:Fnn} (with $n = kp^i$ there) and noticing $F_{kp^i}\equiv0 \pmod{P^{e_i}}$,
    we obtain 
    \begin{equation}  \label{eq:bFkp}
    \begin{split}
        (bF_{kp^i -1})^2&\equiv(-b)^{kp^i}  \equiv \left( \frac{a^2}{4}-\frac{h}{4}P^{\lceil 2e_1/(p-1) \rceil} \right)^{kp^i}  \\ 
        &\equiv \left( (\dfrac{a^2}{2})^{2p^i}-(\dfrac{h}{4})^{p^i}P^{p^i \lceil 2e_1/(p-1) \rceil} \right)^k  \\
        &\equiv (\frac{a}{2})^{2kp^i}    \pmod{P^{e_i}} \quad \textrm{(due to \eqref{eq:piei})}.
    \end{split}    
    \end{equation} 
    In addition, since $\pi(P^{e_i^{\prime}})=p^{i-1}\pi(P^{e_1^{\prime}}) = p^{i-1}\pi(P) = kp^i$, 
    we have $F_{kp^i} \equiv 0 \pmod{P}$ and $F_{kp^i+1} \equiv 1 \pmod{P}$. 
    Noticing $F_{kp^i+1} = aF_{kp^i} + bF_{kp^i - 1}$, we get 
    \begin{equation}  \label{eq:bFkpP}
       bF_{kp^i - 1} \equiv 1 \pmod{P}. 
    \end{equation}
    
    Since $p$ is an odd prime and $P \mid bF_{kp^i -1} - 1$ from \eqref{eq:bFkpP}, we know that $P \nmid bF_{kp^i -1} + 1$. 
    Note that due to $k = \ord_P(a/2)$, we have $(a/2)^{kp^i} \equiv 1 \pmod{P}$. 
    So, we have $P \nmid bF_{kp^i -1} + (a/2)^{kp^i}$.
    Hence, combining this with \eqref{eq:bFkp}, we obtain 
    $$
       bF_{kp^i  -1}\equiv (\frac{a}{2})^{kp^i}  \pmod{P^{e_i}},
    $$
    which, together with $F_{kp^i +1}=aF_{kp^i}+bF_{kp^i -1}\equiv bF_{kp^i - 1} \pmod{P^{e_i}}$, implies that 
    \begin{equation}  \label{eq:Fkpi+}
        F_{kp^i + 1} \equiv  (\frac{a}{2})^{kp^i}  \pmod{P^{e_i}}. 
    \end{equation}
    
    Recall that we have assumed $v_P((a/2)^k -1) = m$. Then, we have 
    $$
       v_P((a/2)^{kp^i} -1) = v_P(((a/2)^{k} -1)^{p^i}) = p^i v_P((a/2)^{k} -1) = mp^i.
    $$
    This means that 
    \begin{equation}  \label{eq:Pmpi}
        (\frac{a}{2})^{kp^i} \equiv  1  \pmod{P^{mp^i}} \quad \textrm{and} \quad  (\frac{a}{2})^{kp^i} \not\equiv  1  \pmod{P^{mp^i +1}}.
    \end{equation}

    Now, for any integer $i \ge 2$, if $mp^i \ge e_i$, from \eqref{eq:Fkpi+} and \eqref{eq:Pmpi} we have 
    $$
      F_{kp^i +1} \equiv 1 \pmod{P^{e_i}}. 
    $$
    Then, noticing $\pi(P^{e_i^{\prime}}) = kp^i$ and the choice of $e_i^{\prime}$, we have 
    $$
      e_i^{\prime} \geq e_i. 
    $$
    Combining this with Lemma~\ref{lem:eiei'}, we obtain 
    $$
      e_i^{\prime} = e_i. 
    $$
    So,  we get   
    \begin{equation}  \label{eq:mpi>}
       e_{i}^{\prime} = e_i = \frac{(p^i-1)e_1}{p-1}  \quad \textrm{when $mp^i \ge e_i$}.     
    \end{equation}
    
    In addition, for any integer $i \ge 2$, if $mp^i < e_i$, from \eqref{eq:Fkpi+} and \eqref{eq:Pmpi} we have 
    $$
        F_{kp^i +1} \equiv  1  \pmod{P^{mp^i}} \quad \textrm{and} \quad  F_{kp^i +1} \not\equiv  1  \pmod{P^{mp^i +1}}. 
    $$
    Then, noticing $\pi(P^{e_i^{\prime}}) = kp^i$ and the choice of $e_i^{\prime}$, we have 
     \begin{equation}  \label{eq:mpi<}
       e_{i}^{\prime} = mp^i \quad \textrm{when $mp^i < e_i$}.     
    \end{equation}
    Combining \eqref{eq:mpi>} with \eqref{eq:mpi<} and noticing $e_i = \frac{(p^i-1)e_1}{p-1}$, we get  
    $$
         e_{i}^{\prime} = \min\left\{ mp^i,   \frac{(p^i - 1)e_1}{p-1} \right\}.
    $$
    This completes the proof of Part (d).

    Part (e). 
    Similar as the beginning of the proof of Part (d), we can obtain $\dfrac{a^2}{b} \not\in \F_q$. 
    So, by Lemma~\ref{lem:ei} we know that for any integer $i \ge 1$, the integer $e_i$ exists.
    
    Note that in this case the characteristic of $\F_q$ is $p=2$. 
    Since $\Delta = a^2 + 4b = a^2 \pmod{P}$, we have $P \mid a$. 
    Noticing $F_2 = a$, by definition we have $\alpha(P)=2$ and $e_1 = v_P(a)$. 
    Besides, in this case we have assumed that $k=\ord_P(b)$ and $m = v_P(b^k -1)$. 
    
    Since $\alpha(P^{e_1}) = \alpha(P) = 2$ and 
    $$
       U^2 = bI + aU \equiv bI \pmod{P^{e_1}}, 
    $$
    by \eqref{eq:betas} we obtain 
    $$
       \pi(P) = \alpha(P) \ord_P(b) = 2\ord_P(b) = 2k
    $$
    and 
    $$
        \pi(P^{e_1}) = \alpha(P^{e_1}) \ord_{P^{e_1}}(b)    = 2\ord_{P^{e_1}}(b), 
    $$
    which, together with  $\pi(P) = \pi(P^{e_1})$ (due to $e_1 = e_1^{\prime}$), implies that $ \ord_{P^{e_1}}(b) = k$. 
    So, we have 
    $$
       v_P(b^k -1) \ge e_1,  \quad \textrm{that is,}  \quad m \ge e_1. 
    $$
    
    In the following, we complete the proof by considering two subcases: Case (e1) and Case (e2). 
    In Case~(e1), we assume $m > e_1$. In Case~(e2), we assume $m = e_1$. 
    
    Case (e1): $m > e_1$. 
    Since $U^2 = bI + aU$, it is easy to prove by induction that for any integer $i \ge 2$, we have 
    \begin{equation}   \label{eq:U2i}
    \begin{split}
       U^{2^{i}}  & = \sum_{r=0}^{i-1} a^{2^i -2^{r+1}}b^{2^r} I + a^{2^i - 1}U  \\
          & = (a^{2^i -2}b + a^{2^i - 2^2}b^2 +  \cdots + a^{2^{i-1}}b^{2^{i-2}} + b^{2^{i-1}})I + a^{2^i - 1}U. 
    \end{split}    
    \end{equation}
    In addition, since $\F_q[x]/P\F_q[x]$ is a finite field of chatracteristic 2, we know that $k = \ord_P(b)$ is an odd integer. 
    Then, noticing $v_P(a) = e_1$ and $m > e_1$,  from \eqref{eq:U2i} we obtain that
    \begin{equation*}
        U^{2^i k} \equiv b^{2^{i-1}k}I \equiv I \pmod{P^{2^{i-1}e_1}}, 
    \end{equation*}
    and 
    \begin{equation*}
    \begin{split}
        U^{2^i k} & \equiv (a^{2^{i-1}}b^{2^{i-2}} + b^{2^{i-1}})^k I \equiv (b^{2^{i-1}k} + ka^{2^{i-1}}b^{2^{i-1}(k-1)+2^{i-2}})I \\
        & \equiv (b^{2^{i-1}k} + a^{2^{i-1}}b^{2^{i-1}(k-1)+2^{i-2}})I \\
        & \equiv (1 + a^{2^{i-1}}b^{2^{i-1}(k-1)+2^{i-2}})I \not\equiv I \pmod{P^{2^{i-1}e_1 +1}}. 
    \end{split}
    \end{equation*}
    Then, noticing $\pi(P^{e_i^{\prime}}) = 2^{i-1}\pi(P) = 2^i k$ and the choice of $e_i^{\prime}$, we get that 
    $$
       e_i^{\prime} = 2^{i-1}e_1. 
    $$
    This completes the proof of Case (e1).

     Case (e2): $m = e_1$. Note that in this case, $v_P(a) = e_1$, $v_P(b^k - 1) = e_1$, and $e_i = (2^i -1)e_1$ for any $i \ge 2$. 
     Then,    from \eqref{eq:U2i} we deduce that for any $i \ge 2$, 
     \begin{equation*}
     \begin{split}
        U^{2^{i}k} & \equiv \Big(b^{2^{i-1}} + \sum_{r=0}^{i-2} a^{2^i -2^{r+1}}b^{2^r} \Big)^k I \\
        & \equiv \Big( b^{2^{i-1}k} + b^{2^{i-1}(k-1)} \sum_{r=0}^{i-2} a^{2^i -2^{r+1}}b^{2^r} \Big)I \\
        & \equiv I + \Big( (b^{k} -1)^{2^{i-1}} + b^{2^{i-1}(k-1)} \sum_{r=0}^{i-2} a^{2^i -2^{r+1}}b^{2^r} \Big)I \\
        & \equiv I + b^{-2^{i-1}}\Big( b^{2^{i-1}}(b^{k} -1)^{2^{i-1}} + b^{2^{i-1}k} \sum_{r=0}^{i-2} a^{2^i -2^{r+1}}b^{2^r} \Big)I \\
        & \equiv I + b^{-2^{i-1}}\Big( b^{2^{i-1}}(b^{k} -1)^{2^{i-1}} +  \sum_{r=0}^{i-2} a^{2^i -2^{r+1}}b^{2^r} \Big)I  \pmod{P^{e_i}}. 
     \end{split}
     \end{equation*}
   Hence, noticing $e_i^{\prime} \le e_i$, $\pi(P^{e_i^{\prime}}) = 2^i k$ and the choice of $e_i^{\prime}$, we obtain 
     \begin{equation*}
     \begin{split}
       e_{i}^{\prime} & = \min\left\{e_i, \quad v_P \big( b^{2^{i-1}}(b^k - 1)^{2^{i-1}} + \sum_{r=0}^{i-2}a^{2^i - 2^{r+1}}b^{2^r} \big)\right\} \\ 
                             & = \min\left\{(2^i -1)e_1, \quad v_P \big( b^{2^{i-1}-1}(b^k - 1)^{2^{i-1}} + \sum_{r=0}^{i-2}a^{2^i - 2^{r+1}}b^{2^r -1} \big)\right\}.         
     \end{split}
     \end{equation*}
     
    We claim that 
    \begin{equation}   \label{eq:ci}
        c_i = b^{2^{i-1}-1}(b^k - 1)^{2^{i-1}} + \sum_{r=0}^{i-2}a^{2^i - 2^{r+1}}b^{2^r -1},     
    \end{equation}
    and so 
    $$
       e_{i}^{\prime} = \min\{(2^i -1)e_1, \ v_P(c_i)\}. 
    $$
    Hence, it remains to prove \eqref{eq:ci}. 
    
   When $i = 2$, the right-hand side of \eqref{eq:ci} is $b(b^k - 1)^2 + a^2 = a^2 + bc_1^2$, which exactly equals to $c_2$. 
   Assume that \eqref{eq:ci} holds for $c_{i-1}$ for any $i \ge 3$.  
   By definition and by the induction hypothesis, we obtain
   \begin{align*}
       c_i & = a^{2^i - 2} + bc_{i-1}^2 \\
       & = a^{2^i - 2} + b\big(b^{2^{i-2}-1}(b^k - 1)^{2^{i-2}} + \sum_{r=0}^{i-3}a^{2^{i-1} - 2^{r+1}}b^{2^r -1}\big)^2 \\
       & = a^{2^i - 2} + b^{2^{i-1}-1}(b^k - 1)^{2^{i-1}} + \sum_{r=0}^{i-3}a^{2^{i} - 2^{r+2}}b^{2^{r+1} -1} \\     
       & = a^{2^i - 2} + b^{2^{i-1}-1}(b^k - 1)^{2^{i-1}} + \sum_{r=1}^{i-2}a^{2^{i} - 2^{r+1}}b^{2^r -1} \\      
       & = b^{2^{i-1}-1}(b^k - 1)^{2^{i-1}} + \sum_{r=0}^{i-2}a^{2^{i} - 2^{r+1}}b^{2^r -1},                       
   \end{align*}
   as claimed in \eqref{eq:ci}.
   This completes the proof. 
\end{proof}

In Part (e) of Theorem~\ref{thm:ei'} when $m = e_1$, one can see that the formula is not enough visible. 
We try to make it more clear in the following theorem. 

\begin{theorem}
\label{thm:ei'2}
    Assume that $a$ and $b$ are not both in $\F_q$. 
Assume further that $p=2$,  $\Delta \equiv 0 \pmod{P}$ and $e_1^{\prime} = e_1$, and $v_P(b^k -1) = e_1$, where $k=\ord_P(b)$. 
Let $g_1 = a/P^{e_1}$, $g_2 = (b^k -1)/P^{e_1}$, and $m_2 = v_P(g_1^2 + bg_2^2)$. 
Moreover, we define the integers $m_3, m_4, \ldots$ and the polynomials $g_3, g_4, \ldots$, by following the process (note that the process might be finite or infinite): 
\begin{itemize}
\item  if $m_2= e_1$, then define $g_3 = (g_1^2 + bg_2^2)/P^{e_1}$, and $m_3 = v_P(g_1^6 + bg_3^2)$;

\item for any $i \ge 3$, assume that we have defined the integers $m_3, \ldots, m_i$ and polynomials $g_3, \ldots, g_i$, 
if $m_i = e_1$, then define $g_{i+1} = (g_1^{2^i -2} + bg_i^2)/P^{e_1}$, and $m_{i+1} = v_P(g_1^{2^{i+1} -2} + bg_{i+1}^2)$.
\end{itemize}
Then, if the above integers $m_2, \ldots, m_j$ exist but $m_{j+1}$ does not exist for some integer $j \ge 2$, then we have that 
for any integer $i$ with $2 \le i \le j$, 
 \begin{equation*}
     e_i^{\prime} = 
            \begin{cases}
                 (2^i - 1)e_1 & \text{if $i < j$}, \\
                \min\{(2^j - 1)e_1, \ (2^j -2)e_1 + m_j\}  & \text{if $i=j$}, 
            \end{cases}
 \end{equation*}
and for any $i > j$, 
$$
  e_i^{\prime} = 2^{i-j} e_j^{\prime}.
$$
 \end{theorem}
\begin{proof}
Note that since the integers $m_2, \ldots, m_j$ exist but $m_{j+1}$ does not exist, we know that the above process stops after computing $m_j$. 
Since the integer $m_{j+1}$ does not exist, by definition we must have $m_j \ne e_1$. 

By Part (e) of Theorem~\ref{thm:ei'}, for any $i \ge 2$ we have 
\begin{equation}  \label{eq:ei'}
      e_{i}^{\prime} = \min\{(2^i -1)e_1, \ v_P(c_i)\}, 
\end{equation}
where $c_i = a^{2^i -2} + bc_{i-1}^2$ and $c_1 = b^k - 1$. 

We claim that for any $i \ge 2$, if the integer $m_i$ is well-defined, then we have 
\begin{equation}  \label{eq:ci=}
   c_i = P^{(2^i - 2)e_1} (g_1^{2^i - 2} + bg_i^2). 
\end{equation}
Indeed, for $i=2$, by definition 
$$
c_2 = a^2 + bc_1 = a^2 + b (b^k -1) = P^{2e_1}(g_1^2 + bg_2^2), 
$$
as claimed in \eqref{eq:ci=}. 
Assume that \eqref{eq:ci=} holds for $c_{i-1}$ for any $i \ge 3$ and the integer $m_i$ is well-defined.  
Then, by definition we have $m_{i-1} = e_1$, and 
\begin{equation*}
g_i = (g_1^{2^{i-1}-2} + bg_{i-1}^2)/P^{e_1}, 
\end{equation*}
which, together with the induction hypothesis, gives 
\begin{equation}  \label{eq:ci-1}
  c_{i-1} = P^{(2^{i-1} - 2)e_1} (g_1^{2^{i-1} - 2} + bg_{i-1}^2) = P^{(2^{i-1} - 1)e_1} g_i. 
\end{equation}
By definition and by \eqref{eq:ci-1}, we obtain 
\begin{equation*}
  c_i = a^{2^i -2} + bc_{i-1}^2 = a^{2^i -2} + bP^{(2^i - 2)e_1}g_i^2 = P^{(2^i - 2)e_1} (g_1^{2^i - 2} + bg_i^2). 
\end{equation*}
This completes the proof of \eqref{eq:ci=}. 

Hence, combining \eqref{eq:ei'} with \eqref{eq:ci=}, we get that for any $i \ge 2$, if the integer $m_i$ is well-defined, then 
\begin{equation}  \label{eq:ei'=}
      e_{i}^{\prime} = \min\{(2^i -1)e_1, \ (2^i - 2)e_1 + v_P(g_1^{2^i - 2} + bg_i^2)\}.
\end{equation}

Now, for any $i \ge 2$, if $i < j$ (this implies $j \ge 3$), since the integer $m_j$ is well-defined, 
we have $m_i = e_1$, that is, $v_P(g_1^{2^i - 2} + bg_i^2) = e_1$, which together with \eqref{eq:ei'=} gives 
\begin{equation}  \label{eq:ei'i<j}
      e_{i}^{\prime} = \min\{(2^i -1)e_1, \ (2^i - 2)e_1 + e_1\} = (2^i -1)e_1.
\end{equation}

Since $m_j = v_P(g_1^{2^j - 2} + bg_j^2)$, by \eqref{eq:ei'=} we have 
\begin{equation}  \label{eq:ei'i=j}
      e_{j}^{\prime} = \min\{(2^j -1)e_1, \ (2^j - 2)e_1 + m_j)\}.
\end{equation}

In addition, using \eqref{eq:ei'} and \eqref{eq:ci=} and noticing $m_j = v_P(g_1^{2^j - 2} + bg_j^2) \ne e_1$, we deduce that 
\begin{equation}  \label{eq:ej+1'}
\begin{split}
  e_{j+1}^{\prime} & = \min\{(2^{j+1}-1)e_1, \ v_P(c_{j+1})\} \\
  & = \min\{(2^{j+1}-1)e_1, \ v_P(a^{2^{j+1}-2} + bc_j^2)\} \\
  & = \min\{(2^{j+1}-1)e_1, \ v_P(a^{2^{j+1}-2} + bP^{(2^{j+1}-4)e_1}(g_1^{2^j -2} + bg_j^2)^2)\} \\  
  & = \min\{(2^{j+1}-1)e_1, \ \min\{(2^{j+1}-2)e_1, \ (2^{j+1}-4)e_1 + 2m_j\}\}. 
\end{split}
\end{equation}

Now, if  $m_j > e_1$, then by \eqref{eq:ei'i=j}, we have 
\begin{equation*}
e_j^{\prime} = (2^j -1)e_1, 
\end{equation*}  
and by \eqref{eq:ej+1'}, we have 
\begin{equation*}
e_{j+1}^{\prime}  = (2^{j+1}-2)e_1 = 2e_j^{\prime}, 
\end{equation*}  
and so, using Lemma~\ref{lem:ei'p} repeatedly we obtain that  for any $i > j$, 
\begin{equation*}
     e_i^{\prime} = 2 e_{i-1}^{\prime}.
\end{equation*}
If $m_j < e_1$, then similarly we obtain 
\begin{equation*}
e_j^{\prime} = (2^j - 2)e_1 + m_j,  
\end{equation*}  
and 
\begin{equation*}
e_{j+1}^{\prime}  = (2^{j+1}-4)e_1 + 2m_j = 2e_j^{\prime}, 
\end{equation*}  
and thus, using Lemma~\ref{lem:ei'p} repeatedly we get that  for any $i > j$, 
\begin{equation*}
  e_i^{\prime} = 2 e_{i-1}^{\prime}.
\end{equation*}

Therefore, we always have that for any $i > j$, 
\begin{equation*}
  e_i^{\prime} = 2 e_{i-1}^{\prime}, 
\end{equation*}
which yields directly 
\begin{equation}  \label{eq:ei'i>j}
  e_i^{\prime} = 2^{i-j} e_{j}^{\prime}.
\end{equation}

Finally, \eqref{eq:ei'i<j}, \eqref{eq:ei'i=j} and \eqref{eq:ei'i>j} give the desired result. 
\end{proof}

We present two examples for Theorem~\ref{thm:ei'2} to suggest that the process 
in the theorem can be finite (Example~\ref{exam1}) or infinite (Example~\ref{exam2}). 

\begin{example} 
\label{exam1}
Let $p=2, a =x^5+x^3+x, b=x^2+1, P=x^2 + x +1$. Then, we have $\Delta \equiv 0 \pmod{P}$, 
$k=3$,  $v_P(b^3-1)=2$, $m_2 = 2$, $m_3 = 0$, $e_1 =e_1^{\prime} = 2$, 
and for any integer $i \ge 2$, $e_i = 2^{i+1} -2$ and $e_i^{\prime} = 2^{i-1} \cdot 3$. 
\end{example}

\begin{example}
\label{exam2}
Let $p=2, a =x^{12} + x^9 + x^8 +x^7 + x^6 +x^5 + x^4 + x, b=x^3 + x, P=x^4 + x^3 + 1$. 
Then, we have $\Delta \equiv 0 \pmod{P}$, $k=3$,  $v_P(b^3-1)=1$, $e_1 =e_1^{\prime} = 1$,  
and for any integer $i \ge 2$, $m_i = 1$, $e_i = 2^i -1$ and $e_i^{\prime} = 2^i - 1$. 
\end{example}

Now, we first compute $e_1$ and $e_1^{\prime}$ by direct computations, and then using Theorems~\ref{thm:ei'} and \ref{thm:ei'2}, we compute the integers $e_i^{\prime}, i \ge 2$. 
Then, noticing the choices of the integers $e_i^{\prime}$,  
we directly obtain the following result for computing the value $\pi(P^e)$. 

\begin{theorem}  \label{thm:pi}
   Assume that $a$ and $b$ are not both in $\F_q$. 
Put $e_0^{\prime} = 0$. 
Then, for any integer $e \ge 1$, we have that  $e_{i-1}^{\prime} < e \le e_i^{\prime}$ for some integer $i \ge 1$, and moreover, 
$\pi(P^e) = \pi(P^{e_i^{\prime}}) = p^{i-1}\pi(P)$.
\end{theorem}

\section{Computing $\alpha(P)$ and $\pi(P)$}
\label{sec:P}

Using Theorems~\ref{thm:alpha} and \ref{thm:pi}, we essentially reduce the problem of computing $\alpha(P^e)$ and $\pi(P^e)$ to that 
 of computing $\alpha(P)$ and $\pi(P)$. 
 
 Since $P$ is an irreducible polynomial, the sequence $\{F_n\}$ modulo $P$ is in fact a sequence over the finite field $\F_{q^d}$, where $d = \deg(P)$. 
Recall that $f(X)=X^2 - aX - b$ is the characteristic polynomial of the sequence $\{F_n\}$.  
Let $\lambda_1, \lambda_2$ be the two roots of $f$ mod $P$, and we denote their orders by $\ord_P(\lambda_1), \ord_P(\lambda_2)$ respectively. 
By  \cite[Theorem 8.44]{LN}, we know that 
\begin{equation}  \label{eq:piP}
   \pi(P) = \lcm[\ord_P(\lambda_1), \ord_P(\lambda_2)]. 
\end{equation}
Since $\deg(f) = 2$ and $\gcd(b, P)=1$, we have that $\lambda_1, \lambda_2$ are both non-zero and in $\F_{q^{2d}}$. 
So, by \eqref{eq:piP} and noticing $\alpha(P) \mid \pi(P)$,  we have 
$$
  \pi(P) \mid q^{2d} - 1 \quad \textrm{and} \quad \alpha(P) \mid q^{2d} - 1.
$$ 
The following theorem improves the above divisibility property, which is an analogue of a result in the integer case (see \cite[Theorem 3]{Ren}).

Recall that $p$ is the characteristic of $\F_q$, and $\Delta = a^2 + 4b$. 

\begin{theorem} \label{the:piP1}
   Let $p$ be an odd prime, and  let $d=\deg(P)$.  Then, we have: 
\begin{itemize}
    \item [(a)] if $\Delta$ is a non-zero quadratic residue modulo $P$, then $\alpha(P) \mid q^d-1$ and $ \pi(P) \mid q^d-1$; 
    
    \item [(b)] if $\Delta$ is a quadratic non-residue modulo $P$, then $\alpha(P) \mid q^d+1$ and $ \pi(P) \mid (q^d+1)\ord_P(-b)$; 

    \item [(c)] if $\Delta \equiv 0 \pmod{P}$, then $\alpha(P)=p$ and $\pi(P)=p \cdot \ord_P(2^{-1}a) $.
\end{itemize}
\end{theorem}

\begin{proof}
    Part (a). 
    Since $\Delta$ is a non-zero quadratic residue modulo $P$ and $\gcd(b, P)=1$, we know that 
    $f$ has two distinct non-zero roots modulo $P$, say $\lambda_1$ and $\lambda_2$, in $\F_{q^d}$. 
    Then, there exists an invertible matrix $T$ such that $U  \equiv TDT^{-1} \pmod{P}$
    where $D =
    \begin{bmatrix}
        \lambda_1&0\\
        0&\lambda_2
    \end{bmatrix}$.
    Thus, 
    $$
    U^{q^d-1}  \equiv TD^{q^d-1}T^{-1} \equiv T
    \begin{bmatrix}
        \lambda_1^{q^d-1}&0\\
        0&\lambda_2^{q^d-1}
    \end{bmatrix}
    T^{-1} \equiv I  \pmod{P},
    $$
    which by definition implies that 
    $$\alpha(P) \mid q^d-1 \quad \textrm{and} \quad  \pi(P) \mid q^d-1.$$

    Part (b).
     Since $\Delta$ is a quadratic non-residue modulo $P$, 
    we have that the two distinct roots of $f$ modulo $P$ are in $\F_{q^{2d}}$ but not in $\F_{q^d}$. 
    Let $\gamma$ be one of the two roots. Then, the other root is $\gamma^{q^d}$. 
    Since $\gamma$ and $\gamma^{q^d}$ are the two roots of the equation $X^2-aX-b=0$, we have 
    $$
      \gamma \gamma^{q^d} =  \gamma^{q^d+1} \equiv -b \pmod{P}, 
    $$
    and also
    $$
       \gamma^{q^d q^{d+1}}=(\gamma^{q^{d+1}})^{q^d} \equiv (-b)^{q^d} \equiv -b \pmod{P}.
    $$
    Thus, for some invertible matrix $T$, 
    $$U^{q^d+1} \equiv T
    \begin{bmatrix}
        \gamma^{q^d+1}&0\\
        0&\gamma^{q^d(q^d+1)}
    \end{bmatrix}
    T^{-1} \equiv T
    \begin{bmatrix}
        -b&0\\
        0&-b
    \end{bmatrix}
    T^{-1} \equiv -bI \pmod{P} ,
    $$
    which implies 
    $$(U^{q^d+1})^{\ord_P(-b)} \equiv I \pmod{P}.$$
    Hence, we obtain 
    $$
       \alpha(P) \mid q^d+1 \quad \textrm{and} \quad  \pi(P) \mid (q^d+1)\ord_P(-b).
    $$

    Part (c).
     Since $\Delta \equiv 0 \pmod{P}$ and $p$ is odd, 
    we get that $f$ modulo $P$ has a repeated root (that is, $2^{-1}a  \pmod{P}$). 
    Then, for some invertible matrix $T$, we have   $U=TJT^{-1}$, 
    where 
    $J =
    \begin{bmatrix}
        2^{-1}a &1\\
        0 & 2^{-1}a
    \end{bmatrix}$. 
    
    For any integer $n \ge 1$, we have 
    $J^n =
    \begin{bmatrix}
        (2^{-1}a)^n & n(2^{-1}a)^{n-1}\\
        0 & (2^{-1}a)^n
    \end{bmatrix}.
    $
    Then, when $n=p$, we have $U^p \equiv (2^{-1}a)^{p} I \pmod{P}$. 
    So, we obtain $\alpha(P)=p$.

    In addition, for any integer $n \ge 1$, we have 
    $U^{pn} \equiv (2^{-1}a)^{pn} I \pmod{P}$.
    So, we get that $U^{pn} \equiv I \pmod{P}$ if and only if $\ord_P(2^{-1}a) \mid n$. 
    Hence, we have 
    $$\pi(P)=p\cdot \ord_P(2^{-1}a).$$
\end{proof}

Now, we consider the case when $p=2$. 
In this case, $f(X) = X^2 + aX + b$, and $\Delta = a^2 + 4b = a^2$. 
$f$ has two distinct roots modulo $P$ if and only if $\Delta \not\equiv 0 \pmod{P}$, that is, $a \not\equiv 0 \pmod{P}$. 
To determine whether the roots are in $\F_{q^d}$, the situation is quite different from the case of odd characterisitc. 

Assume $a \not\equiv 0 \pmod{P}$. 
Define the polynomial $g(X) = X^2 + X + b/a^2$. 
Note that $g(X) = f(aX)/a^2$. 
So, the roots of $f$ modulo $P$ are in $\F_{q^d}$ if and only if the roots of $g$ modulo $P$ are in $\F_{q^d}$. 
Recall that $q=p^{l}$. So, $q^d = p^{dl} = 2^{dl}$. 
Define the trace function $\Tr$ of $\F_{q^d}$: 
$$
\Tr(x)=\sum_{i=0}^{dl-1}x^{2^i}, \quad x \in \F_{q^d}. 
$$
Then, by \cite[Theorem 1]{BRS} we know that 
the roots of $f$ modulo $P$ are in $\F_{q^d}$ if and only if $\Tr(b/a^2) \equiv 0 \pmod{P}$.  
Moreover, one can write down the solutions by using the formulas given in \cite{Chen}. 

Now, we are ready to establish the divisibility result for the case when $p=2$. 

\begin{theorem}
Let $p=2$, and  let $d=\deg(P)$. Then, we have: 
\begin{itemize}
    \item [(a)] if $a \not\equiv 0 \pmod{P}$ and $\Tr(b/a^2) \equiv 0 \pmod{P}$ , then $\alpha(P) \mid q^d-1$ and $ \pi(P) \mid q^d-1$; 
    
    \item [(b)] if $a \not\equiv 0 \pmod{P}$ and $\Tr(b/a^2) \not\equiv 0 \pmod{P}$ , then $\alpha(P) \mid q^d+1$ and $ \pi(P) \mid (q^d+1)\ord_P(b)$; 

    \item [(c)] if $a \equiv 0 \pmod{P}$, then $\alpha(P)=2$ and $\pi=2 \ord_P(b)$. 
\end{itemize}

\end{theorem}

\begin{proof}
    Parts (a) and (b). 
    We have known that the roots of $f$ modulo $P$ are in $\F_{q^d}$ if and only if $\Tr(b/a^2) \equiv 0 \pmod{P}$.  
    Then, the desired result follows from the same arguments as in proving Parts (a) and (b) of Theorem~\ref{the:piP1}. 
    
    Part (c). 
    Since $a \equiv 0 \pmod{P}$, 
    we get $U \equiv 
    \begin{bmatrix}
        0&1\\
        b&0
    \end{bmatrix} \pmod{P}$.

    Observing $U^2\equiv
    \begin{bmatrix}
        b&0\\
        0&b
    \end{bmatrix} \equiv bI \pmod{P}$, we have $\alpha(P) = 2$, and also 
    \begin{equation*}
        U^n \equiv \left\{
        \begin{aligned}
            &\begin{bmatrix}
                b^{n/2}&0\\
                0&b^{n/2}
            \end{bmatrix} , & n \text{ is even}, \\
            &\begin{bmatrix}
                0&b^{(n-1)/2}\\
                b^{(n+1)/2}&0
            \end{bmatrix} , & n \text{ is odd}.
        \end{aligned}
        \right.
        \end{equation*}
    Hence, we obtain 
    $$\pi(P)=2 \ord_P(b).$$
\end{proof}

\section{Properties of $\beta(M)$}
\label{sec:beta}
In this section, we study the function $\beta(M)$, which is defined to be $\pi(M)/\alpha(M)$.  
Recall that we have assumed $\gcd(b, M)=1$. 
We have also mentioned before that $\beta(M)$ is exactly the number of zero terms in one (minimal) period of the sequence $\{F_n\}$ modulo  $M$. 

The following theorem is an analogue of a result in the integer case (see \cite[Theorem 4]{Ren}).

\begin{theorem}  \label{thm:beta0}
We have:
    \begin{itemize}

        \item [(a)] $\beta(M) \mid 2 \ord_M(-b)$;
        
        \item [(b)] $\pi(M)=(1 \text{ or } 2) \cdot \lcm [\alpha(M),\ord_M(-b)]$.
    \end{itemize}
\end{theorem}

\begin{proof}
    (a) By definition, we have $U^{\alpha(M)} \equiv sI \pmod{M}$ for some polynomial $s$. 
    So, we have $ (\det U)^{\alpha(M)} \equiv s^2 \pmod{M}$, which together with  $\det U = -b$ gives 
    \begin{equation}  \label{eq:s2}
        (-b)^{\alpha(M)} \equiv s^2 \pmod{M}. 
    \end{equation}
    Now, we obtain 
    $$
       s^{2\ord_M(-b)} \equiv (-b)^{\alpha(M)\ord_M(-b)} \equiv 1 \pmod{M}, 
    $$
    which, together with  $\beta(M) = \ord_M(s)$ by \eqref{eq:betas}, yields the desired result.

    (b) Note that 
    $$
       \ord_M(s^2) = \frac{\ord_M(s)}{\gcd(2, \ord_M(s))}, 
    $$
    and
    $$   
       \ord_M((-b)^{\alpha(M)}) = \frac{\ord_M(-b)}{\gcd(\alpha(M), \ord_M(-b))}. 
    $$
    Combining this with \eqref{eq:s2}, we get 
    \begin{equation*}
       \frac{\ord_M(s)}{\gcd(2, \ord_M(s))} = \frac{\ord_M(-b)}{\gcd(\alpha(M), \ord_M(-b))}. 
    \end{equation*}
    Since $\beta(M) = \ord_M(s)$, we have 
    \begin{equation*}
       \frac{\beta(M)}{\gcd(2, \beta(M))} = \frac{\ord_M(-b)}{\gcd(\alpha(M), \ord_M(-b))}. 
    \end{equation*}
    Then, we obtain 
    \begin{align*}
       \pi(M) = \alpha(M) \beta(M) & = \gcd(2, \beta(M)) \cdot \frac{\alpha(M)\ord_M(-b)}{\gcd(\alpha(M), \ord_M(-b))} \\
       & = \gcd(2, \beta(M)) \cdot \lcm[\alpha(M),\ord_M(-b)], 
    \end{align*}
    which, together with $\gcd(2,\beta(M))=1$ or $2$, gives the desired result. 
\end{proof}

From Part (a) of Theorem~\ref{thm:beta0}, we know that if $b = -1$, then $\beta(M)$ is always 1 or 2. 
Part (b) of Theorem~\ref{thm:beta0} implies that it is quick to compute $\pi(M)$ if $\alpha(M)$ and $\ord_M(-b)$ are given. 

Recall that the analogue of the function $\beta$ in the integer case, denoted by $\omega$, has been studied in \cite{Ren}. 
By Theorem 5 of \cite{Ren}, we know that for any fixed prime $\ell$, the value $\omega(\ell^e)$ remains unchanged when $e$ is sufficiently large. 
However, this does not always hold for the function $\beta$, unless $e_i = e_i^{\prime}$ for any $i \ge 1$. 

For example, assume that $\Delta = a^2 + 4b = 0$ and  $a, b$ are not both in $\F_q$. 
Then, by Part (a) of Theorem~\ref{thm:alpha0} and Part (a) of Theorem~\ref{thm:ei'} we know that $\beta(P^e)$ goes to the infinity as $e$ tends to the infinity. 

The following theorem says that
 usually the function $\beta$ can be bounded from below and from the above. 

\begin{theorem}
Assume that $\Delta \ne 0$ and $\dfrac{a^2}{b} \not\in \F_q$. 
Assume further that $e_1/ e_1^{\prime} \le p^k$ for some integer $k \ge 1$. 
Then, for any integer $e > e_{k+2}$, we have 
$$
\beta(P) \le \beta(P^e) \le p^{k+1}\beta(P). 
$$
\end{theorem}
\begin{proof}
Since $\Delta \ne 0$ and $\dfrac{a^2}{b} \not\in \F_q$ (this implies that $a$ and $b$ are not both in $\F_q$), by Lemmas~\ref{lem:ei} and \ref{lem:ei'} we know that all the integers $e_1, e_2, \ldots$ and $e_1^{\prime}, e_2^{\prime}, \ldots$ exist.

For any fixed integer $e > e_{k+2}$,  we first have $e_{i-1}^{\prime} < e \le e_{i}^{\prime}$ for some positive integer $i$. 
Then, by definition we obtain 
\begin{equation}  \label{eq:piPj}
  \pi(P^e) = \pi(P^{e_{i}^{\prime}}) = p^{i-1} \pi(P).
\end{equation}

Now, on one hand, by Lemma~\ref{lem:eiei'} we get 
$$
e \le e_i. 
$$
Since $e > e_{k+2}$, we must have $i > k+2$.
On the other hand, using Theorems~\ref{thm:ei} and \ref{thm:ei'}, we have 
$$
e > e_{i-1}^{\prime} \ge p^{i-2} e_1^{\prime} = p^{i-2} e_1 \cdot \frac{e_1^{\prime}}{e_1} \ge p^{i-k-2}e_1 > \frac{p^{i-k-2} - 1}{p-1} e_1 \ge e_{i-k-2}, 
$$
So, we obtain 
$$
e_{i-k-2} < e \le e_i, 
$$
which implies 
$$
\alpha(P^{e_{i-k-1}}) \le   \alpha(P^e) \le \alpha(P^{e_i}). 
$$
Then, we get 
\begin{equation}  \label{eq:alphaPj}
  p^{i-k-2}\alpha(P) \le \alpha(P^e) \le p^{i-1}\alpha(P).
\end{equation}

Finally, combining \eqref{eq:piPj} with \eqref{eq:alphaPj}, we obtain 
$$
\beta(P^e) = \frac{\pi(P^e)}{\alpha(P^e)} \ge \frac{p^{i-1}\pi(P)}{p^{i-1}\alpha(P)} = \beta(P), 
$$
and 
$$
\beta(P^e) = \frac{\pi(P^e)}{\alpha(P^e)} \le \frac{p^{i-1}\pi(P)}{p^{i-k-2}\alpha(P)} = p^{k+1}\beta(P).
$$
This completes the proof. 
\end{proof}

\section*{Acknowledgement}
For the research, Z. Chen and C. Wei were supported by the Guangdong College Students' 
Innovation and Entrepreneurship Training Program (No. S202210574094); 
M. Sha was supported by the Guangdong Basic and Applied Basic Research Foundation (No. 2022A1515012032).


\begin{thebibliography}{99}

\bibitem{BRS}
E. R. Berlekamp, H. Rumsey and G. Solomon, \textit{On the solution of algebraic equations over finite fields}, Inf. Control, 10 (1967), 553--564.

\bibitem{Car}
R. D. Carmichael, \textit{On sequences of integers deﬁned by recurrence relations}, Quart. J. Math., 41 (1920), 343-372.

\bibitem{Chen}
C.-L. Chen, \textit{Formulas for the solutions of quadratic equations over GF$(2^m)$}, IEEE Trans. Inform. Theory, 28 (1982), 792-794. 

\bibitem{Cox}
D. A. Cox, \textit{Galois theory}, 2 ed., Wiley, Hoboken, 2012.  


\bibitem{Eng}
H. T. Engstrom, \textit{On sequences deﬁned by linear recurrence relations}, Trans. Amer. Math. Soc., 33 (1931), 210-218. 

\bibitem{EPSW}
G. Everest, A. van der Poorten, I. Shparlinski and T. Ward, \textit{Recurrence Sequences},
Amer. Math. Soc., Providence, RI, 2003.

\bibitem{FHM}
R. Fl\'orez, R. Higuita and A. Mukherjee, \textit{Characterization of the strong divisibility property for generalized Fibonacci polynomials}, 
Integers, 18 (2018), \#A14.

\bibitem{FHR}
R. Fl\'orez, R. Higuita and A. Ram\'\i rez, \textit{The resultant, the discriminant, and the derivative of generalized Fibonacci polynomials}, 
J. Integer Seq., 22 (2019), Article 19.4.4.

\bibitem{FMM}
R. Fl\'orez, N. McAnally and A. Mukherjee, \textit{Identities for the generalized Fibonacci polynomial}, Integers, 18B (2018), \#A2.

\bibitem{FS}
R. Fl\'orez and J. C. Saunders, \textit{Irreducibility of generalized Fibonacci polynomials}, Integers, 22 (2022), \#A69.


\bibitem{Gup}
S. Gupta, P. Rockstroh and F. E. Su, \textit{Splitting ﬁelds and periods of Fibonacci se-quences modulo primes}, Math. Mag., 85 (2012), 130-135. 

\bibitem{Hall}
M. Hall, \textit{Divisors of second-order sequences}, Bull. Amer. Math. Soc., 43 (1937), 78-80. 

\bibitem{Kal}
D. Kalman and R. Mena, \textit{The Fibonacci numbers—exposed}, Math. Mag., 76 (2003), 167-181. 

\bibitem{Li}
H.-C. Li, \textit{On second-order linear recurrence sequences: Wall and Wyler revisited}, Fib. Quart., 37 (1999), 342-349.

\bibitem{LN}
R. Lidl and H. Niederreiter, \emph{Finite fields},  Cambridge University Press, 1997.

\bibitem{Lucas}
E. Lucas, \textit{Th\'{e}orie des fonctions num\'{e}riques simplement p\'{e}riodiques}, Amer. J. Math., 1 (1878), 184-240, 289-321.


\bibitem{Ren}
M. Renault, \textit{The period, rank, and order of the $(a,b)$-Fibonacci sequence mod $m$}, Math. Mag., 86 (2013), 372-380.


\bibitem{Rob}
D. W. Robinson, \textit{The Fibonacci matrix modulo $m$}, Fib. Quart., 1 (1963), 29-36.

\bibitem{Vel}
D. Vella and A. Vella, \textit{Cycles in the generalized Fibonacci sequence modulo a prime}, Math. Mag., 74 (2002), 294-299.

\bibitem{Vel2}
D. Vella and A. Vella, \textit{Calculating exact cycle lengths in the generalized Fibonacci sequence mod-ulo $p$}, Math. Gaz., 90 (2006), 70-76.

\bibitem{Vin}
J. Vinson, \textit{The relation of the period modulo $m$ to the rank of apparition of $m$ in the Fibonacci sequence}, Fib. Quart., 1 (1963), 37-45.

\bibitem{Wall}
D. D. Wall, \textit{Fibonacci series modulo $m$}, Amer. Math. Monthly, 67 (1960), 525-532. 

\bibitem{Ward}
M. Ward, \textit{Prime divisors of second-order recurring sequences}, Duke Math. J., 21 (1954), 607-614.

\end{thebibliography}
\end{document}